\documentclass[12pt,letterpaper, final]{article}
      \usepackage{latexsym}
         \usepackage[reqno, namelimits, sumlimits]{amsmath}
         \usepackage{amssymb, amsfonts}
          \usepackage{amsthm, verbatim}

\newtheorem{Theorem}{Theorem}[section]

\newtheorem{theorem}[Theorem]{Theorem}
\newtheorem{lemma}[Theorem]{Lemma}

\newtheorem{remark}[Theorem]{Remark}

\newtheorem{Definition}[Theorem]{Definition}

\DeclareMathOperator{\dist}{dist} 
 \DeclareMathOperator{\cl}{cl}

\DeclareMathOperator{\diam}{diam}

\newcommand{\R}{{\mathbb R}}
\newcommand{\N}{{\mathbb N}}
\newcommand{\Z}{{\mathbb Z}}

\newcommand{\al}{\alpha}
\newcommand{\be}{\beta}
\newcommand{\ga}{\gamma}

\newcommand{\ep}{\epsilon}

\newcommand{\De}{\Delta}

\newcommand{\om}{\omega}

\newcommand{\Om}{\Omega}
\newcommand{\La}{\Lambda}
\newcommand{\LA}{\Lambda}
\newcommand{\la}{\lambda}

\newcommand{\varep}{\varepsilon}

\newcommand{\codt}{\cdot}

\newcommand{\grad}{\nabla}

\newcommand{\cE}{\mathcal{E}}
\newcommand{\cA}{\mathcal{A}}

\newcommand{\als}{\alpha}

 \title{Equilibria with a nontrivial nodal set and the 
dynamics of  parabolic equations on symmetric domains
}
 \author{Juraj F\" oldes\\
\small IMA, University of Minnesota\\
\small Minneapolis, MN 55455\\
~\\
Peter Pol\'a\v cik\footnote{Supported in part by  NSF Grant
 DMS-1161923  
}  
\\
\small School of Mathematics, University of Minnesota\\
\small Minneapolis, MN 55455
}

\date{\empty}

\begin{document}

\numberwithin{equation}{section}

\maketitle


\begin{abstract}
We consider the Dirichlet problem
\begin{alignat}{2} \label{eqa}
u_t &= \Delta u + f(x, u, \nabla u)+ h(x, t),& \qquad &(x, t) \in \Om
\times (0, \infty)\,, \\  
u &= 0, & \qquad &(x, t) \in \partial\Om \times (0, \infty)\,, \label{bca}
\end{alignat}
on a bounded domain $\Om\subset \R^N$. The domain and
the nonlinearity $f$ are assumed to be 
invariant under the reflection about the
$x_1$-axis, and the function $h$ accounts for
a nonsymmetric decaying perturbation:
 $h(\codt, t)\to 0$ as $t\to\infty$. In
one of our main theorems,  we prove the asymptotic symmetry of each
bounded positive solution $u$ of \eqref{eqa}, \eqref{bca}. 
The novelty of this 
result is that the asymptotic 
symmetry is established even for solutions that are
not assumed uniformly positive. 
In particular, some equilibria of the limit
time-autonomous problem (the problem with $h\equiv 0$) 
with  a nontrivial nodal set  may occur in the $\om$-limit set 
of $u$ and this prevents one from applying common techniques based
 on the method of moving hyperplanes.
The goal of our second main theorem is to classify the positive entire 
solutions of the time-autonomous problem.  We prove that if $U$ is a
positive entire solution, then one of the following applies:
(i) for each $t\in \R$, $U(\codt,t)$ is even in $x_1$ and decreasing in
$x_1>0$, (ii) $U$ is an equilibrium, (iii) $U$ is a connecting orbit
from an equilibrium with a nontrivial nodal set to a set consisting of 
functions which are even in $x_1$ and decreasing in
$x_1>0$, (iv) is a heteroclinic connecting orbit between two
equilibria with a nontrivial nodal set.
\end{abstract}

\ \\ 


\ \\

\section{Introduction}

In this paper, we consider two classes of semilinear parabolic
problems, 
\begin{alignat}{2} \label{eq0}
u_t &= \Delta u + f(x, u, \nabla u),& \qquad &(x, t) \in \Om
\times (0, \infty)\,, \\  
u &= 0, & \qquad &(x, t) \in \partial\Om \times (0, \infty)\,, \label{bc0}
\end{alignat}
and
\begin{alignat}{2} \label{eq1}
u_t &= \Delta u + f(x, u, \nabla u) + h(x, t),& \qquad &(x, t) \in \Om \times (0, \infty)\,, \\
u &= 0, & \qquad &(x, t) \in \partial\Om \times (0, \infty)\,. \label{bc1}
\end{alignat}
Here $\Om\subset \R^N$ is a bounded domain and 
$f:(x,u,p)\mapsto f(x,u,p):\bar \Om\times [0,\infty)\times \R^{N+1}\to\R$ 
is a continuous function, 
which is Lipschitz in $u$ and $p$.  In \eqref{eq1},  $h$ 
is a bounded continuous function which decays to 0 as $t\to
\infty$.  Thus  \eqref{eq1} is an asymptotically
autonomous equation,  \eqref{eq0} being its 
``limit'' autonomous equation.  We only consider nonnegative
solutions:  by not defining $f(x,u,p)$ for  $u<0$, we
postulate that only nonnegative functions can be solutions of 
\eqref{eq0}, \eqref{eq1}. 

We further assume that $\Om$ is symmetric
about the hyperplane 
$$H_0:=\{(x=(x_1,\dots,x_N)\in \R^N:x_1=0\}$$
 and convex in the direction of the $x_1$-axis. Also we make a symmetry
 assumption on $f$ (see (F2) in the next section), 
which makes equation 
\eqref{eq0} equivariant under
the reflections about the hyperplanes parallel to $H_0$.
We then examine the dynamics of global   
 solutions of \eqref{eq0}, \eqref{bc0}, and \eqref{eq1}, \eqref{bc1} 
from the  symmetry point of view. One of our main objectives is to clarify 
whether all  global bounded solutions of \eqref{eq1},
\eqref{bc1} are asymptotically symmetric. 
Toward that goal, we  investigate the dynamics of the limit  
problem \eqref{eq0}, \eqref{bc0}; in particular, we want to understand
how the dynamics is affected by the presence of
 equilibria with nontrivial nodal set.

To spell our goals out, 
let us briefly summarize earlier results pertinent to our study. 
It is well known that positive steady
states of \eqref{eq0}, \eqref{bc0} are symmetric  about $H_0$ and strictly
decreasing with increasing $|x_1|$, see \cite{Berestycki-N, Dancer:symm, Gidas-N-N:bd, 
Li:symm-bd} (for related symmetry results see also 
\cite{Berestycki:survey, Da-L-S,  Kawohl:survey, Ni:survey, Pucci-S}
and references therein).  The steady states which are nonnegative, but
not strictly positive, obviously fail to have the strict monotonicity 
property. However, as  shown in \cite{P:symm-ell}, they
still enjoy the symmetry about $H_0$ and, in addition, they are
reflectionally symmetric within their nodal domains (in particular, the
nodal domains themselves are reflectionally symmetric). Examples 
of nonnegative steady states with a nontrivial nodal set can be found in
 \cite{P:symm-ell, p-Terr}. 
 
For time-dependent solutions of \eqref{eq0}, \eqref{bc0}, 
two kinds of symmetry results are available, one dealing with
\emph{entire solutions,} that is, solutions defined 
for all $t\in\R$, the other one with \emph{global solutions,} 
that is, solutions defined 
for all $t>0$. 

For  entire solutions, the spatial 
symmetry about $H_0$ and monotonicity in $x_1>0$ 
was established in \cite{Babin:symm1, Babin-S}. The hypotheses of this
result include in particular a uniform positivity  condition requiring
 the solutions in question to stay away from zero at each point $x\in \Om$, 
uniformly in time. 

In the case of global solutions, 
the symmetry at all times cannot be established unless  
the initial condition is symmetric. Rather, it has been proved that global
positive solutions are asymptotically symmetric in the sense that all
their limit profiles as $t\to\infty$, or elements of their $\om$-limit sets, are symmetric 
about $H_0$ and monotone nondecreasing in $x_1>0$ (see 
\cite{Babin:symm1, Babin-S, p-He:symm, P:est}, 
related results can be found in 
\cite{Foldes:bd, P:symm-survey, Saldana-W}). 
Similarly as for elliptic equations \cite{Berestycki-N}, 
the symmetry results for parabolic equations concerning 
the entire solutions \cite{Babin-S} as well as the global solutions
\cite{p-Fo:asympt-symm, P:est}, 
have been proved for fully nonlinear equations on nonsmooth 
domains. Moreover, in the results for the parabolic equations,
 the nonlinearities are allowed to depend on $t$, and some asymptotic
 symmetry results can also be established if the equation is
 asymptotically symmetric, not necessarily symmetric at all times 
\cite{F-asy:Rn, p-Fo:asympt-symm}. Some sort of uniform positivity condition 
is always used in the proof of these results and, alongside
the asymptotic symmetry of the solutions,  one also proves
their asymptotic monotonicity with respect to $x_1$ (for $x_1\ge 0$).

This brings us to the main topic of this paper.
We want to address the question whether the asymptotic symmetry of a
solution  $u$ can be proved even if  its $\om$-limit set may possibly
contain  functions which are not monotone in $x_1>0$. 
Consider for example the situation  
when problem  \eqref{eq0}, \eqref{bc0}
possesses nonnegative equilibria with nontrivial nodal sets. 
If $u$ is a positive global solution of \eqref{eq0}, \eqref{bc0} or
\eqref{eq1}, \eqref{bc1}, then, without any uniform positivity
assumption on $u$, some of these equilibria  
can appear in the $\om$-limit set of $u$ (an example where this 
happens can be found in \cite[Example 2.3]{P:est}).  This prevents $u$
from being  asymptotically monotone in $x_1>0$, which makes it
clear that the asymptotic symmetry of $u$ cannot be established by a 
direct application of the moving plane method. In this regard, 
this symmetry problem is quite different from the ones discussed above. 
Also note that even for the time-autonomous equation 
\eqref{eq0}, there is  no obvious gradient structure and 
a given solution may not approach a set of equilibria. 
Therefore, it is not a priori  clear whether the symmetry of equilibria, 
as proved in  \cite{P:symm-ell}, has  
any significance for the asymptotic symmetry 
of general positive solutions.

Letting the asymptotic symmetry problem aside for a while, 
the question of how the equilibria with nontrivial 
nodal sets enter into the asymptotic behavior of positive solutions
is quite interesting itself.
Clearly, the presence of such an equilibrium 
$\varphi$ in  the $\om$-limit set of a solution $u$ means
that at some times $t_k \to \infty$, the function
$u(\cdot,t_k)$ has a ``near-nodal set'' resembling
the nodal set of $\varphi$. If two such equilibria are contained  in 
$\om(u)$, then the graph of $u(\cdot,t)$ 
forms two different  near-nodal patterns and repeatedly transfers
from one to the other, as $t\to\infty$.  
Surprisingly perhaps, we show that this does not
happen. One of our main results, Theorem \ref{thm1}, addresses  
this issue as well as
the asymptotic symmetry problem. 
It says that  if $u$ is a bounded global solution of \eqref{eq1},
\eqref{bc1}, then the 
following alternative concerning its $\om$-limit set holds. 
Either $\om(u)$ consists of functions which are symmetric about $H_0$ and 
monotone nonincreasing in $x_1>0$, or it consists 
of a single nonnegative equilibrium with a nontrivial nodal set.
Since the nonnegative equilibria are all  symmetric about $H_0$, this
result implies the asymptotic symmetry of the solution $u$. Also it
shows that the only way $u$ may fail to be asymptotically monotone 
in $x_1>0$ is that it converges to an equilibrium with a nontrivial nodal
set. In particular, if $u$ has a nontrivial asymptotic nodal pattern,
then the pattern is unique.

Our second main result concerns entire solutions of  
\eqref{eq0}, \eqref{bc0}. Such solutions play a distinguished role in
the global dynamics of \eqref{eq0}, \eqref{bc0}. For example, if a
global attractor of \eqref{eq0}, \eqref{bc0} exists, then it is formed
by entire solutions (see \cite{Hale:bk-diss, Temam:bk}).  
It is also well-known that the $\om$-limit sets 
of solutions of \eqref{eq1}, \eqref{bc1} consist of 
entire solutions of  \eqref{eq0}, \eqref{bc0}. This is 
relevant for the asymptotic symmetry result discussed above. Namely,
the asymptotic symmetry of bounded positive solutions of  
\eqref{eq1}, \eqref{bc1} would be proved if one could establish 
the symmetry of all nonnegative entire solutions 
of  \eqref{eq0}, \eqref{bc0}. 
The problem whether all entire  solutions of \eqref{eq0}, \eqref{bc0}, 
even those which are not monotone in $x_1>0$, are symmetric about $H_0$ 
at all times is open and we do not resolve it in this paper. 
It appears to be much harder than in the case of equilibria. 
In particular, the method of \cite{P:symm-ell}, which works 
well for nonnegative solutions of
elliptic equations,  does not extend to time dependent
solutions of parabolic equations. What we can prove for 
a general nonnegative  entire solution $U$ is  that  
one of the following
possibilities occurs: 
\begin{itemize}
\item[(i)] $U(\cdot,t)$ is 
symmetric about $H_0$ and monotone in $x_1>0$ for each $t\in \R$,
\item[(ii)] $U$ is a nonnegative equilibrium with a nontrivial nodal set,
\item[(iii)] $\om(U)$ consists of functions, which are symmetric about $H_0$ and 
monotone nonincreasing in $x_1>0$, and $\al(U)$ consists 
of a single nonnegative equilibrium with a nontrivial nodal set,
\item[(iv)] $U$ is a connecting orbit between two equilibria, each
  having a nontrivial nodal set.
\end{itemize}
 See Theorem \ref{thm2} and Remark \ref{rmtothm2}(a) 
below for more  precise statements. Here $\al(U)$
stands for the $\al$-limit set of $U$, that is, the set of all limit
profiles of $U(\cdot,t)$ as $t\to-\infty$. 

The above result implies that either  $U$ is 
symmetric about $H_0$ (cases (i) and (ii)) or else $U(\cdot,t)$ converges to
an equilibrium $\psi$ as $t\to-\infty$  
 (cases (iii) and (iv)). Even with this
additional information, it is not clear whether $U$ is 
symmetric. While there are results on the symmetry of the unstable
manifold of a positive equilibrium $\psi$ (see \cite{Babin:symm2,
  p-Ha} or \cite{P:symm-survey}), they depend on comparison arguments
and the positivity of the function 
$-\partial_{x_1}\psi$ in $\{x\in \Om:x_1>0\}$.  No simple modification of
such arguments applies if $\psi$ has a nontrivial nodal set. 

In any case, our theorem  gives an interesting description of the behavior of  
a general  nonnegative entire solution and, although it does not give the
symmetry of all entire solutions, it is sufficient for the proof of  
the  asymptotic symmetry  result discussed above. 
Let us explain briefly how we
use (i)-(iv) in the  proof of the asymptotic symmetry.
Inspired by \cite{p-He:symm}, we consider a functional 
$\La$ arising in  the  process of moving hyperplanes. Roughly speaking, 
for a function $z$, $\La(z)$ measures 
how far to the left one can move the hyperplane
$H_\la=\{x\in\R^N:x_1=\la\}$,   
while preserving a relation between the
graph of $z$ and its reflection through the hyperplane $H_\la$. 
In \cite{p-He:symm}, $\La$ was shown to be 
 decreasing along positive solutions
which are not symmetric from the start.  
Thus, $\La$ can be viewed as a strict Lyapunov functional and from
this viewpoint  the asymptotic symmetry result 
of \cite{p-He:symm} is a manifestation 
of the LaSalle invariance principle. In our more general setting, 
particularly due to the lack of any smoothness of $\Om$, 
we cannot prove that $\La$ is a strict Lyapunov functional, 
however, it serves us in a similar way. Specifically, 
we   prove that if $U$ is a connecting orbit as in  (iv),
then $U$ connects an equilibrium with higher ``energy'' (the
value of  $\La$) to an equilibrium with lower energy
(see Remark \ref{rmtothm2}(d) below). 
This conclusion, combined with  a 
chain-recurrence property of  $\om(u)$,
implies that if an equilibrium $\psi \in \om(u)$ has 
a nontrivial nodal set, 
then necessarily $\om(u)=\{\psi\}$. On the other hand, in view of the
possibilities  
(i)-(iv), if  $\om(u)$ contains no such equilibrium, 
then it consists of entire solutions satisfying (i).  
 This implies the conclusion of  our  
asymptotic symmetry theorem.

We conclude the introduction with a few remarks concerning the
structure of equations \eqref{eq0}, \eqref{eq1}. 
We work with autonomous or asymptotically autonomous equations
mainly because our goal was to understand how the equilibria 
with a nontrivial nodal set fit into the dynamics of nonnegative
solutions. We have chosen to consider semilinear equations, rather than
fully nonlinear equations, in order not to obscure the key ideas 
by additional technicalities. Since the equations do not have  
a gradient structure, the main conceptual difficulties 
 are already present in this study.

The remainder of the paper is organized as follows.
In the next section we state our main results. Their proofs are given
in Sections \ref{proof}, \ref{proof2}. 
Section \ref{prelim} contains preliminary material, including 
a discussion of linear equations  arising in the process 
of moving hyperplanes and some symmetry results from earlier papers. 
Basic monotonicity properties 
of the functional $\La$ along solutions of
nonautonomous symmetric equations are also   derived 
in Section \ref{prelim}.

\section{Main results}
\label{main}
We start by precisely stating our hypotheses.
\begin{itemize}
\item[(D1)] $\Om \subset \R^N$ is a bounded domain, which is
  symmetric with respect to $H_0$ and convex in the direction of the $x_1$-axis.
\item[(D2)] For each $\la > 0$, the set 
\begin{equation*}
\Om_\la := \{x \in \Om : x_1 > \la\}
\end{equation*}
has only finitely many connected components.
\item[(A)] There exist numbers $\varsigma \in (0, 1)$ and $R > 0$
  such that for each $x \in \partial \Om$, $\rho \in (0, R)$,  
one has
\begin{equation*}
|\Om \cap B(x, \rho)| \leq \varsigma |B(x, \rho)| \,,
\end{equation*}
where $B(x, \rho)$ is  the ball of radius $\rho$ centered at $x$ and
$|\codt|$ stands for the Lebesgue measure.  
\end{itemize}

Condition (A) is a minor regularity requirement on $\Om$ which allows
us to use  boundary H\"older estimates on the solutions of
\eqref{eq1}, \eqref{bc1}. The technical condition (D2) is assumed in
order to make the results of \cite{P:symm-ell} applicable.

Concerning the functions $f:(x,u,p)\mapsto f(x,u,p):
\bar \Om\times [0,\infty)\times \R^{N}\to\R$ and $h:\Om \times (0,
\infty)\to \R$ 
we assume the following.

\begin{itemize}

\item[(F1)] (\textit{Regularity})
$f$ is continuous on $\bar \Om\times [0,\infty)\times \R^N$,  
\begin{equation}
  \label{eq:3}
  f \in
C^{\alpha_0}_{\textrm{loc}}(\Om \times [0, \infty)\times \R^N)
\end{equation}
for some $\alpha_0 \in (0,1)$, and $f$ 
is differentiable with bounded derivatives
 with respect to $(u, p)$. In particular, $f$
 is Lipschitz  in  $(u, p)$ uniformly with respect to
$x \in \bar{\Om}$:  there is $\beta_0 >0$ such that
\begin{multline*}
    \sup_{x\in \bar{\Om}} |f(x, u,p)-f(x,\tilde u,\tilde p)|\le
    \beta_0 |(u,p)-(\tilde u,\tilde p)|\qquad \\ 
    (x\in\bar\Om, (u,p),(\tilde u,\tilde p)\in  [0,\infty) \times \R^N)\,.
  \end{multline*} 
\item[(F2)] (\textit{Symmetry}) $f$ is independent of $x_1$ and even
  in $p_1$: 
  \begin{multline*}
  f(x, u,-p_1,p_2,\dots,p_N)=f(x, u,p_1,p_2,\dots,p_N)\\ (x,
u,p_1,p_2,\dots,p_N)\in \bar \Om\times [0,\infty) \times \R^N).
  \end{multline*}
\item[(H)\ ] $h$ is  continuous and bounded,  and 
\begin{equation*}
\lim_{t \to \infty} \|h (\cdot, t)\|_{L^\infty(\Om)} = 0 \,.
\end{equation*}
\end{itemize}

When dealing with nonlinear problems \eqref{eq0},
  \eqref{bc0} and \eqref{eq1},
  \eqref{bc1} (or problem \eqref{tdc} introduced below), we
  always consider classical solutions. In particular a global solution
  of \eqref{eq1}, \eqref{bc1} is a function $u \in
  C^{2,1}_{\textrm{loc}}(\Om \times (0, \infty)) \cap C(\bar{\Om}
  \times [0,\infty)) $ satisfying the equation and the boundary
  condition everywhere. We shall consider global solutions which are
  bounded (this simply means that $u$ is a bounded function on
  $\Om\times (0,\infty)$).

To formulate our main results, we need to introduce some notation. For
$\la\in \R$, we set 
\begin{equation*}
\begin{aligned}
H_{\la} &:= \{x \in \R^N : x_1 = \la \} ,  \\
\ell &:= \sup \{\la : \Om \cap H_{\la} \neq \emptyset \} \,, \\
\Om_{\la} &:= \{x \in \Om : x_1 > \la \},
\end{aligned}
\end{equation*}
and let $P_{\la} : \R^N \to \R^N$ be the reflection about
 $H_{\la}$, that is, 
 \begin{equation*}
   P_{\la}x := (2\la - x_1, x') \quad (x = (x_1, x') \in \R^N).
 \end{equation*}
For any function 
$z \in C(\bar{\Om})$, we define 
$V_\la z : \Om_\la \to \R$ by 
\begin{equation*}
  V_\la z(x) := z(P_\la x)-z(x).
\end{equation*}
Since $\Omega$ is convex in $x_1$, $V_\lambda z$ is well defined for
any $\lambda \geq 0$.

Let  $C_0(\bar{\Om})$ stand for  
the space of continuous functions on $\bar{\Om}$ 
vanishing on $\partial \Om$  
equipped with the supremum norm.
We define a functional $\La : C_0(\bar{\Om}) \to [0, \ell]$ by 
\begin{equation*}
\La (z) := \inf\{\la \in (0, \ell] : V_\mu z(x) \geq 0 \quad(x \in
\Om_{\mu},\ \mu > \la) \}.
\end{equation*}
Since $\Om_\ell = \emptyset$,  the set on the right hand side trivially
contains $\la=\ell$, thus  $\La(z) \in [0, \ell]$ is well defined for
each $z \in C_0(\bar{\Om})$. 

\begin{remark}\label{nng}
\textnormal{
It is clear from the definition of $\La(z)$ that   
$z$ is nonincreasing in $x_1$ in $\Om_{\La(z)}$. 
}
\end{remark}

 We denote by $E$ the set of equilibria  (time-independent solutions) 
of \eqref{eq0}, \eqref{bc0}. We allow ourselves a harmless ambiguity 
and view the equilibria as 
elements of $C_0(\bar{\Om})$ or as functions on
$\Om\times \R$ constant in time, depending on the context.  
Set
\begin{equation*}
  \begin{aligned}
E_0 &:= \{z \in E : \La(z) = 0\} \,,\\
  E_+ &:= E\setminus E_0= \{z \in E : \La(z) > 0\} \,.
\end{aligned}
\end{equation*}
Recall that we only consider nonnegative solutions ($f(x,u,p)$ 
is not defined for $u<0$). In particular,   
all equilibria of \eqref{eq0}, \eqref{bc0} are nonnegative
hence, by \cite{P:symm-ell}, they are all even in  $x_1$.
By Remark \ref{nng}, the set $E_0$ 
consists of the equilibria which are nonincreasing 
in $x_1>0$; in fact, each of them is either identically equal to zero or
strictly positive and decreasing in $x_1>0$. 
It also follows from \cite{P:symm-ell} that  
$E_+$ is the set  of all equilibria whose nodal set
in $\Om$ is nontrivial (different from $\Om$ and $\emptyset$).

As usual, we shall discuss the asymptotic behavior of a bounded
solution $u$ of \eqref{eq1}, \eqref{bc1} in terms of its $\omega$-limit set. 
For that we first note that the semi-orbit $\{u(\cdot, t) : t \in
(1, \infty)\}$ is precompact in $C_0(\bar{\Om})$.
This is a consequence of Arzel\`a-Ascoli theorem and the following
H\"older estimate 
\begin{equation}\label{eq:hol1}
\sup_{ \substack{x, \bar{x} \in \bar{\Om}, x \neq \bar{x} \\
      t, \bar{t} \in [s, s+1], t\neq \bar{t}\\
      s \in [1, \infty)}}
\frac{|u(x, t) - u(\bar{x}, \bar{t})|}{|x - \bar{x}|^\alpha + |t - \bar{t}|^{\alpha/ 2}} <\infty\,.    
\end{equation}
The fact that any bounded solution satisfies
\eqref{eq:hol1}  for  some  $\al\in (0,1)$
is proved in \cite[Proposition 2.7]{P:est} for
fully nonlinear equations, including  \eqref{eq0} as a special case;
the proof is really just a summary of  well-known interior and 
boundary H\"older estimates (condition (A) is needed for the boundary
estimates; the interior estimates hold irrespectively  of condition
\eqref{eq:3} which  is not assumed \cite{P:est}).    
One just needs to verify
an extra assumption in \cite[Proposition 2.7]{P:est}, which in our case
requires the boundedness of the function
 $(x, t) \mapsto f(x, 0, 0) + h(x, t)$ on $\Om\times
(0,\infty)$. This requirement is clearly met due to hypotheses
(F1) and (H). 

Once the precompactness of  $\{u(\cdot, t) : t \in
(1, \infty)\}$ has been established, it follows by standard results that 
the $\omega$-limit set of $u$ in $C_0(\bar{\Om})$, that is, the set 
\begin{equation*}
\omega (u) := \bigcap_{t > 0} \cl_{C_0(\bar{\Om})} \{u(\cdot, s): s \geq t\},
\end{equation*} 
is nonempty, compact, connected, 
and it attracts the semi-orbit of $u$: 
\begin{equation} \label{oas}
\lim_{t \to \infty} \dist_{C_0(\bar{\Om})}(u(\cdot, t), \omega(u)) = 0\,.
\end{equation}

Our main theorem concerning the bounded solutions of \eqref{eq1},
\eqref{bc1} can
 now be stated as follows.

\begin{theorem}\label{thm1}
Assume (D1), (D2), (A), (F1), (F2), (H), and let $u$ be a bounded global
 solution of \eqref{eq1}, \eqref{bc1}. Then 
each $z \in \omega(u)$ is even in $x_1$:
\begin{equation*}
  z(x_1, x') = z(-x_1, x')\quad ((x_1, x') \in \Om).
\end{equation*}
Moreover, either $\omega(u) = \{z\}$ for some 
$z \in E_+$, or each $z \in \omega(u)\setminus \{0\}$ 
is (strictly) decreasing in $x_1$ on $\Om_0$.
\end{theorem}

We next consider entire solutions  of \eqref{eq0}, \eqref{bc0}. We
usually use symbol $U$ for an entire solution. Thus $U$ satisfies (in
the classical sense) the problem
\begin{equation}\label{lim}
\begin{aligned}
    U_t &= \Delta U + f(x, U, \nabla U)\,,&& \qquad (x, t) \in
    \Om\times \R \,,\\ 
   U&=0 ,&& \qquad (x, t) \in \partial\Om \times \R \,.  \\
\end{aligned}
\end{equation} 
Denote 
\begin{align*}
\mathcal{A} := \{U : \ &\text{$U$ is a bounded (non-negative)
 entire  solution of \eqref{lim}} \} \,.
\end{align*}
The equilibria of \eqref{lim} are of course examples of entire solutions
in $\cA$, that is, $E\subset \cA$. 

As in the case of  global solutions of \eqref{eq1}, \eqref{bc1}, 
H\"older estimates from \cite[Proposition 2.7]{P:est}
give the following H\"older estimate for each bounded entire solution
$U$ of \eqref{lim}:
\begin{equation}\label{eq:hol1U}
\sup_{ \substack{x, \bar{x} \in \bar{\Om}, x \neq \bar{x} \\
      t, \bar{t} \in [s, s+1], t\neq \bar{t}\\
      s \in \R}}
\frac{|U(x, t) - U(\bar{x}, \bar{t})|}{|x - \bar{x}|^\alpha + |t -
  \bar{t}|^{\alpha/ 2}} <\infty\,.     
\end{equation}
Hence the orbit $\{U(\cdot, t): t \in \R\}$
is precompact in $C_0(\bar{\Om})$.
Defining the $\alpha$ and $\omega$-limit sets of $U$ by
\begin{align}
\alpha (U) := \bigcap_{t \leq 0} \cl_{C_0(\bar{\Om})} \{U(\cdot, s): s \leq t\} \,,\label{defom}\\
\omega (U) := \bigcap_{t \geq 0} \cl_{C_0(\bar{\Om})} \{U(\cdot, s): s \geq t\} \,,\label{defal}
\end{align}
we obtain by standard results that 
$\alpha(U)$ and $\omega(U)$  are nonempty, compact, connected
sets in $C_0(\bar{\Om})$, which attract the orbit of $U$ is the
following sense:
\begin{equation*}
\lim_{t \to \infty} \dist_{C_0(\bar{\Om})}(U(\cdot, t), \omega(U)) = 0,\quad 
\lim_{t \to -\infty} \dist_{C_0(\bar{\Om})}(U(\cdot, t), \alpha(U)) = 0\,.
\end{equation*}

Here is our main result for the entire solutions of \eqref{lim}.
\begin{theorem}\label{thm2}
Assume (D1), (D2), (A), (F1), (F2),  and let $U \in \mathcal{A}$. 
Then exactly one of  the following possibilities occurs:
\begin{enumerate}
\item[(i)] $\La(U(\cdot, t)) = 0$ for each $t \in \R$,
\item[(ii)] $U \in E_+$,
\item[(iii)] $\al(U) = \{\xi_\ast\}$ for some $\xi_\ast \in E_+$ and $\La(z) = 0$ for each $ z\in \omega(U)$ ,
\item[(iv)] $\al(U) = \{\xi_\ast\}$ and \, $\omega(U) = \{\xi^\ast\}$ 
for some $\xi_\ast, \xi^\ast \in E_+$ with $\La(\xi^\ast) < \La(\xi_\ast)$.
\end{enumerate}
If $f(\cdot, 0, 0) \ge 0$, then (i) is the case.
\end{theorem}

\begin{remark}\label{rmtothm2}
{\rm  
  \quad (a) \ Note that if  (i) holds, then  
   $U$ is symmetric (even) in
  $x_1$. To see this, first recall that (i) means that
  \begin{equation}
    \label{eq:6}
    U(P_0x,t)-U(x,t)\ge 0\quad(x\in   \Om_0,\ t\in \R).
  \end{equation}
  Consider now the  
   function
  $\tilde U:=U(P_0\cdot,\cdot)$. Clearly,
  $\tilde U \in \mathcal{A}$, hence  one 
  of the possibilities (i)-(iv) of
  Theorem \ref{thm2} 
  applies to $\tilde U$; we claim that (i) does. Indeed, if not, then  
  $\al(\tilde U) = \{z\}$ for some $z\in E_+$. But then also
  $\al(U) = \{z\}$, because $z\circ P_0=z$  by
  \cite{P:symm-ell}. However, this contradicts 
  Theorem \ref{thm2} (if (i) holds for $U$, then none of the other
  possibilities can occur). So (i) holds for both $U$ and $\tilde U$,
  which implies that \eqref{eq:6} holds together with the opposite 
  inequality, and therefore $U$ is even in $x_1$. 

\quad (b) \ 
If (iv) holds,  then
  $U$ is a positive heteroclinic solution between two equilibria 
in $E_+$. We do not have an example of an equation where such a
heteroclinic solution occurs. It cannot occur if, for example,
$\Om$ is convex in all variables, for in that case 
$E_+$ contains at most one element (see \cite{P:symm-mult}).  
 On the other hand, it is not difficult
to find examples with a heteroclinic connection 
from an equilibrium in $E_+$ to an equilibrium in $E_0$. 
We sketch an example in dimension $N=1$. Take
 $\Om=(-3\pi,3\pi)$ and let 
$f(u)=u-1$ for $u\le 2$ and $f(u)<0$ for $u\ge 3$. Then 
$E_+$ contains the equilibrium $\xi(x)=1+\cos x$. This equilibrium is
unstable and its fast unstable manifold (an invariant manifold
tangent at $\xi$ to a positive function) contains an entire
solution $U$ monotonically increasing in time and 
such that $U(\codt,t)\to\xi$ as $t\to-\infty$.  
This solution is positive and
bounded (by the condition  $f(u)<0$ for $u\ge 3$)
and its limit   as
$t\to\infty$ is a strictly positive equilibrium, hence an element of
$E_0$. This illustrates that the possibility (iii) can occur.
Possibility (i) or (ii) occurs, for example, if $U$ is an
equilibrium in $E_0$ or $E_+$, respectively.

\quad (c) \  The fact that none of the conditions (ii)-(iv) can hold if 
$f(\cdot, 0, 0) \ge 0$ follows from the strong 
comparison principle: each equilibrium either vanishes identically or
is strictly positive in $\Om$. In particular, $E_+=\emptyset$. 

\quad (d) \  Theorem \ref{thm2} shows that unless  
$U$ is an equilibrium or is 
symmetric, the value of 
 $\La$ on $\om(U)$ is strictly smaller than its value on $\al(U)$. 
In this regard,  $\La$ behaves as a strict Lyapunov functional.
}
\end{remark}

\section{Linear equations and moving hyperplanes}
\label{prelim}
This section has four parts. In Subsections \ref{linear},
\ref{linmove} we recall some useful estimates for solutions of
linear parabolic equations and show how linear equations arise
in the process of moving hyperplanes. In Subsection \ref{basicprop} we
use the estimates for linear equations to derive basic properties
of the  functional $\La$. Finally, in Subsection \ref{symm-nonaut},
we recall two  symmetry results concerning symmetric 
 equations with time-dependent
nonlinearities.

In this section, as in  the whole paper, 
$\Om$ is a fixed domain satisfying conditions (D1), (D2),
and (A).

\subsection{Linear equations}
\label{linear}
We use the following standard notation. For  a bounded
set $G$ in $\R^N$ or $\R^{N+1}$,  $\diam G$ denotes the
diameter of $G$ and $|G|$ for the Lebesgue measure of $G$ (if it is
measurable). By $B(x, r)$ we denote 
the open ball in $\R^N$ centered at $x$ with radius $r$ and 
symbols  $f^{+}$ and $f^{-}$ stand for the positive
and negative parts of a function $f$:  $f^{\pm}:=(|f|\pm f)/2\ge 0$.
For a domain $D \subset \R^N$, 
we define the inner radius of $D$ by
$$
\textrm{inrad}(D) := \sup \{\rho > 0 : B(x_0, \rho) \subset D \textrm{ for some } x_0 \in D\}\,,
$$ 
and if $D$ is an open set, we let $\textrm{inrad}(D)$ stand for the
infimum of inner radii of the connected components of $D$.

For any open bounded $S \subset \R^{n+1}$ and any
bounded, continuous function $f : S \to \R$ define
$$
 [f]_{p, S} := \left(\frac{1}{|S|} \int_{S} |f|^p dx\, dt \right)^{\frac{1}{p}} \qquad (p > 0)\,.
$$

\begin{Definition}{\rm Given an open set
$Q \subset\R^{N + 1}$, and 
 $\beta_0>0$, we say that a
differential operator $L$  belongs to
$\cE(\beta_0, Q)$ if 
\begin{equation*}
L(x,t)= \De +  
\sum_{k = 1}^N b_{k}(x, t) \frac{\partial}{\partial x_k} + c(x, t),
\end{equation*}
where $b_k$, $c$ are measurable functions defined
on $Q$ such that
\begin{alignat*}{2}
|b_k(x,t)|, |c(x,t)| &\leq \beta_0& \qquad &((x, t) \in Q, k =1,\dots,N)\,.
\end{alignat*}
}
\end{Definition}
Given an open set  $G \subset \Om$ and $\tau < T$,  consider the
linear parabolic problem
\begin{align}
v_t &= L(x, t) v + h(x, t), &\qquad &(x, t) \in G \times (\tau, T) \,, \label{prob-lin-e}\\
  v &= 0,  &\qquad &(x, t) \in \partial G  \times (\tau, T) \,,\label{prob-lin-bc}
\end{align}
where $L \in \cE(\beta_0, G \times (\tau, T))$ for some
 $\beta_0 > 0$ and $h \in L^\infty (G \times (\tau, T))$.  
We say  that $v$ is a  supersolution of \eqref{prob-lin-e},
\eqref{prob-lin-bc}  if    
$v \in W^{2,1}_{N + 1, loc}(G \times (\tau, T)) \cap C(\bar{G} \times
[\tau, T))$,  $v\ge 0$ on $\partial G  \times [\tau, T)$,  
and \eqref{prob-lin-e} holds almost
everywhere in $G \times (\tau, T)$ with `=' replaced by `$\geq$'. 
We say $v$ is a subsolution of \eqref{prob-lin-e}, \eqref{prob-lin-bc} if $-v$
is a supersolution and we say $v$ is
a solution of \eqref{prob-lin-e}, \eqref{prob-lin-bc} if it is both supersolution  and subsolution. 

In addition to the standard  maximum  principle, we shall also
use the following estimate (see \cite{Cabre:ABP, Krylov:bk,
  Lieberman:bk, Tso:ABP}). 

\begin{theorem}\label{max:principle}
If $v$ is a supersolution of \eqref{prob-lin-e}, \eqref{prob-lin-bc}, then 
$$
 \|v^-(\cdot, t) \|_{L^\infty(G)} \leq   C^\ast (\|v^-(\cdot, \tau) \|_{L^\infty(G)}  + \|h^{-}\|_{L^{\infty}(G \times (\tau, t))}
  ) \qquad (t \in (\tau, T))\,, 
$$ where $C^\ast = C^\ast(N,   {\beta_0}, T - \tau)$ is a positive
constant. 
\end{theorem}

The following result is the maximum principle on small domains. For the proof see 
\cite{Foldes:bd, P:est}. 

\begin{lemma} \label{zac}
There exists $\delta = \delta (N,   \beta_0)$ such that 
if $|G| < \delta$ and $v$ is a super-solution of  \eqref{prob-lin-e}, \eqref{prob-lin-bc} with $h \equiv 0$, then 
\begin{equation*}
\| v^{-}(\cdot, t)\|_{L^\infty(G)} \leq 
2 e^{-(t - \tau)} \| v^{-}(\cdot, \tau)\|_{L^\infty(G)} \qquad (t \in (\tau, T)) \,. 
\end{equation*}
\end{lemma}

For the proof of the next theorem see \cite{P:est}.

\begin{theorem}\label{imp}
Given any $\rho > 0$, $d > 0$, and $\theta > 0$, there 
exist positive constants 
$\delta = \delta(N,\diam \Om, \beta_0, \rho)$, $p = p(N, \diam \Om,
\beta_0, d, \theta, \rho)$, 
and  
$\tilde{\mu} = \tilde{\mu}(N, \diam \Om, \beta_0, d, \theta, \rho)$ 
with the following properties. If  
$D \subset G \subset \Om$ are open sets satisfying
\begin{equation*}
\textrm{inrad}(D) > \rho, \quad |G \setminus \bar{D}| < \delta, \quad
\dist (\bar{D}, \partial G) \geq d \,, 
\end{equation*}
if $v$ is a super-solution of  \eqref{prob-lin-e}, \eqref{prob-lin-bc} with $T = \infty$ and
$f \equiv 0$, and if  
\begin{equation}\label{amu}
\begin{aligned}
v(x, t) &> 0 \quad ((x, t) \in \bar{D} \times (\tau, \tau + 8\theta)), \\
\|v^-(\cdot, \tau)\|_{L^\infty(G \setminus \bar{D})} &\leq \tilde{\mu}
[v]_{p, D_0 \times (\tau + \theta, \tau + 2\theta)}\,, 
\end{aligned}
\end{equation}
for each connected component $D_0$ of $D$, then the following
statements hold true: 
\begin{align}\label{pvd}
v(x, t) &> 0 \qquad ((x, t) \in \bar{D} \times [\tau, \infty))\,, \\ 
\label{uvd}
\|v^{-}(\cdot, t)\|_{L^\infty (G)} 
&\leq 2 e^{-(t - \tau)} \|v^{-}(\cdot, \tau)\|_{L^\infty (G)}\quad (t > \tau)\,. 
\end{align}
\end{theorem}
Notice that out of the quantities $d$, $\theta$, and $\rho$, 
the constant  $\delta$ depends only by $\rho$, whereas 
the constants $p$ and $\tilde \mu$ also depend on 
$d$ and~$\theta$.

\subsection{From nonlinear to linear  equations}
\label{linmove}
Let us now recall how linear problems of the form \eqref{prob-lin-e}, \eqref{prob-lin-bc}
arise in the process of moving hyperplanes. Below we 
need to apply this process to problem \eqref{lim} as well as to 
a transformation  of \eqref{lim} which is no longer time-autonomous. 
Therefore, we introduce   the following more general problem:

\begin{equation}\label{tdc}
 \begin{alignedat}{2} 
U_t &= \Delta U + g(t, x, U, \nabla U),& \qquad &(x, t) \in \Om
\times \R\,, \\  
U &= 0, & \qquad &(x, t) \in \partial\Om \times \R \,. 
\end{alignedat}
\end{equation}

We assume that the function   $g:(t,x,u,p)\mapsto g(t,x,u,p):
\R\times \bar \Om\times [0,\infty)\times \R^{N}\to\R$ 
satisfies the following two conditions similar to (F1), (F2)
((G1) requires slightly less regularity than (F1)): 

\begin{itemize}

\item[(G1)] (\textit{Regularity}) $g$ is continuous in all variables and 
  Lipschitz  
  in  $(u, p)$:  there is $\beta_0 >0$ such that
\begin{multline*}
    \sup_{x\in \bar{\Om}} |g(t,x, u,p)-g(t,x,\tilde u,\tilde p)|\le
    \beta_0 |(u,p)-(\tilde u,\tilde p)|\qquad \\ 
    (t\in\R,\ x\in\bar\Om, (u,p),(\tilde u,\tilde p)\in  [0,\infty)
    \times \R^N)\,. 
  \end{multline*}
\item[(G2)] (\textit{Symmetry}) $g$ is independent of $x_1$ and even
  in $p_1$.
\end{itemize}
Also we assume the following boundedness condition
\begin{itemize}
\item[(G3)] The function $g(t,x,0,0)$ is bounded on $\R\times \Om$. 
\end{itemize}
Condition (G3) guarantees that  the H\"older estimate 
\eqref{eq:hol1U}  applies to any bounded 
entire solution $U$ of \eqref{tdc}
(see \cite[Proposition 2.7]{P:est}). Consequently, the trajectory of
$U$ and its $\al$ and $\om$-limit sets  have the properties
discussed in Section \ref{main} (see the paragraph containing
 \eqref{defom}, \eqref{defal}).

Denote 
\begin{equation*}
\mathcal{A}^\ast := \{U : \textrm{$U$ is a bounded
 nonnegative entire solution of \eqref{tdc}}\}\,.
\end{equation*}

Given $U \in \mathcal{A}^\ast$ and $\la\in [0,\ell)$, 
 define $U^\la : \bar{\Om}_\la \times \R \to \R$ 
by $U^\la(x, t) := U(P_\la x, t)$.  
By (G2)
\begin{equation*}
\partial_t U^\la = \Delta U^\la + g(t, x, U^\la, \nabla U^\la), 
\quad (x, t) \in \Om_\la \times \R\,.
\end{equation*}
Hence, the function $w^\la : \bar{\Om}_\la \times \R \to \R$,  
\begin{equation} \label{dfw}
w^\la(x, t) := U^\la(x, t) - U(x, t) 
\end{equation}
satisfies
\begin{equation}\label{ref2}
  \partial_t w^{\la} = \Delta w^\la  + g(t, x, U^{\la}, \nabla U^{\la})   - g(t, x, U, \nabla U), \qquad (x, t) \in \Om_{\la} \times \R.
\end{equation}
Using the Hadamard formula, we can rewrite \eqref{ref2}
as 
\begin{equation}\label{ref}
\begin{aligned}
  \partial_t w^{\la} &= L^\la(x, t) w^\la, \qquad (x, t) \in \Om_{\la} \times \R\,,
\end{aligned}
\end{equation}
where $L^{\la} \in \cE(  \beta_0, \Om_\la
\times \R)$, with 
 $\be_0$ as in (G1).

Also, since $U \geq 0$ in $\Om$, $w^\la$ satisfies 
\begin{equation}\label{bcd}
w^{\la} (x, t)  \geq 0 \qquad ((x, t) \in \partial \Om_\la \times \R) \,.
\end{equation}
Hence, $w^\la$ is a supersolution of \eqref{prob-lin-e}, \eqref{prob-lin-bc}, with $G=\Om_\la$
and $L=L^\la$.

We shall also encounter different 
linear equations associated with \eqref{tdc}. 
For example, if $U$, $\tilde U$ are two solutions of 
\eqref{tdc}, then $w=U-\tilde U$
 is a  solution of a linear
problem 
\eqref{prob-lin-e}, \eqref{prob-lin-bc} on $\Om \times \R$ with $L\in \mathcal{E}(\be_0, \Om
\times \R)$.

\subsection{Basic properties of the functional $\La$}
\label{basicprop}
We now use the estimates from 
Subsection \ref{linear} to examine the behavior of the
functional $\La$ along entire solutions of the nonautonomous problem 
\eqref{tdc}. 

We assume that the nonlinearity  $g$  satisfies conditions (G1)-(G3). Of
course, all the results proved here apply to the bounded entire
solutions of the more specific problem
\eqref{lim}.

\begin{lemma}\label{zcc}
There is $\varepsilon_0 > 0$ 
such that for any $U \in \mathcal{A}^*$ one has
\begin{equation*}
  \La(U(\cdot, t)) < \ell -
\varepsilon_0\quad (t\in\R).
\end{equation*} 
\end{lemma}

\begin{proof}
Take $\delta > 0$ as in Lemma \ref{zac} and fix $\varepsilon_0 > 0$ such that $|\Om_{\la}| < \delta$ for any 
$\la \in (\ell - \varepsilon_0, \ell)$. Let $w^\la$ be defined as in \eqref{dfw}. Then by Lemma \ref{zac} one has
\begin{multline*}
\|(w^\la)^-(\cdot, t)\|_{L^\infty(\Om_\la)} \leq 2 e^{-(t - \tau)}
\|(w^\la)^-(\cdot, \tau)\|_{L^\infty(\Om_\la)} \\ 
(\tau \leq t, \la \in (\ell - \varepsilon_0, \ell)) \,.
\end{multline*}
Taking  $\tau \to -\infty$ and using the 
boundedness of $U$, we obtain  $w^\la (\cdot, t) \geq 0$ for any
$\la \in (\ell - \varepsilon_0, \ell)$, which proves the lemma. 
\end{proof}

\begin{lemma}\label{l:dec}
For any $U \in \mathcal{A}^*$, the function $t \mapsto \La(U(\cdot, t))$
is nonincreasing and 
\begin{align}
  &\La(z)\le \lim_{t \to\infty}\La(U(\codt,t))\quad(z\in
  \om(U)),\label{limos}\\
&\La(z)\le \lim_{t \to-\infty}\La(U(\codt,t))\quad(z\in \al(U)).\label{limoa}
\end{align}
\end{lemma}
\begin{proof}
Given any $\tau \in \R$, denote $\la_0 := \La(U(\cdot, \tau))$ and fix
an arbitrary $\la \in (\la_0, \ell)$. If $w^\la$ is as in \eqref{dfw}, 
then $w^\la$ satisfies \eqref{ref}, \eqref{bcd}
and, by our choice of $\la$, it also satisfies
\begin{equation}\label{pos}
w^\la (x, \tau) \geq 0  \quad (x \in \Om_\la).
\end{equation}
By the maximum principle, $w^\la(\cdot, t) \geq 0$ for each $t\geq \tau$. 
Since $\la \in (\la_0, \ell)$ 
was arbitrary, we obtain that 
$\La(U(\cdot, t)) \leq \la_0$ for any $t \geq \tau$ and  
the monotonicity property follows. 

The monotonicity implies that
relation \eqref{limos} is equivalent to 
\begin{equation}\label{equivlimos}
  \La(z)\le \La(U(\codt,t))\quad(z\in
  \om(U),\ t\in \R).
\end{equation}
To prove \eqref{equivlimos}, fix $t\in\R$, $z\in
  \om(U)$,  and  $\la\ge \La(U(\codt,t))$. 
Then, for some sequence $(t_n)_{n \in \N}$ with $t_n>t$, $t_n\to \infty$,
 we have $U(\codt,t_n)\to z$. Consequently,
 $V_\la z = \lim_{n \to \infty} V_\la U(\codt,t_n)\ge 0$ on $\Om_\la$,
 since $\La(U(\codt,t_n))\le \La(U(\codt,t))\le \la$. This proves
 \eqref{equivlimos}. 

 Now fix an arbitrary $z\in \al(U)$. If $\La(z)=0$, then the relation in  
\eqref{limoa} holds trivially. Assume $\La(z)>0$. By the definition of
$\La$, arbitrarily close to   
$\La(z)$ there is $\la<\La(z)$ such that 
$V_{\la}z(x_{\la}) < 0$ for some  $x_{\la} \in \Om_{\la}$. 
Since $z\in \als(U)$,  by a choice of a large negative 
$\tau$ we can make  $U(\codt, \tau)$ so close to   
$z$  that $V_{\la} U(x_{\la}, \tau) < 0$. Consequently, $\La (U(\cdot,
\tau))>\la$ and by the monotonicity $\La (U(\cdot, t))>\la$ for each $t \leq \tau$. 
This proves that $\La(z)\le  \lim_{t \to-\infty}\La(U(\codt,t))$, as desired. 
\end{proof}

\begin{lemma}\label{l:ref-2-moved}
Let $U \in \mathcal{A}^*$ and $\la_0\in (0,\ell)$. Then   
the following two statements are valid:
\begin{itemize}
\item[(i)] If for some $\tau_0\in \R$ one has
$\La(U(\cdot, \tau_0))\le \la_0$  and
$V_{\la_0}U(\codt,\tau_0)\not\equiv 0$ on 
each connected component   of $\Om_{\la_0}$, then  there exists
$\varepsilon_0 > 0$ such that $\La(z) \leq \la_0 - \varepsilon_0$ for
each $z \in \omega(U)$. 
\item[(ii)] If for some $z\in \al(U)$ one has 
  $V_{\la}z> 0$ on $\Om_{\la}$ for each $\la\in [\la_0,\ell)$, 
then  there exists
$\varepsilon_0 > 0$ such that $\La(U(\cdot, t)) 
\leq \la_0 - \varepsilon_0$ for
each $t\in\R$. 
\end{itemize}
\end{lemma}

\begin{proof}  
In this proof we apply Theorem \ref{imp} in much the same way as in
\cite[Proof of Lemma 4.3]{P:est}.

 The proofs of (i) and (ii) use
similar arguments. The following is a common part to both proofs.  
By (D2), $\Om_{\la_0}$ has finitely many connected components, and therefore 
 $\rho :=  \textrm{inrad} (\Om_{{\lambda_0}})/2 > 0$. To
this $\rho$ (and $\be_0$ as in (G1)), there is
$\delta > 0$ as in Theorem \ref{imp}.  
Fix an open set $D$ such that $\bar{D}\subset \Om_{{\lambda_0}}$, 
$D \cap M$ is a domain for each connected component $M$ of $\Om_{{\lambda_0}}$,
$\textrm{inrad} (D) > \rho$, and $|\Om_{{\lambda_0}} \setminus D| <
\delta/{2}$. Then for sufficiently small $\varepsilon_0 > 0$ one has 
$|\Om_{\la} \setminus D| < \delta$ for each $\la \in (\la_0 -
\varepsilon_0, \la_0)$. Below we  assume that $\varepsilon_0>0$ has this
property, but we may need to make it even smaller.
Denote $d := \dist (D, \partial \Om_{{\lambda_0}})$ and observe  that 
$d\leq \dist (D, \partial \Om_{\la})$ for each 
$\la \in (\la_0 - \varepsilon_0, \la_0)$.

For  any $\la\ge 0$, let $w^{\la}$ be as in \eqref{dfw}. 
Recall that $w^{\la}$ satisfies \eqref{ref}
 \eqref{bcd}. 

We now prove statement (i). Since $\la_0\ge \La(U(\codt,\tau_0))$,
we have
\begin{equation}\label{pos0}
w^{\la_0} (x, \tau_0) \geq 0  \quad (x \in \Om_{\la_0}).
\end{equation}
The maximum principle and the assumption in
(i) therefore imply 
that  $w^{\la_0} > 0$ on $\Om_{\la_0} \times (\tau_0, \infty)$. 
Fix any $\tau > \tau_0$ and denote
\begin{equation*}
 r_1 := \frac{1}{2} \inf_{\bar{D}} w^{\la_0}(\cdot, \tau) > 0 \,.
\end{equation*}
Then, by  continuity, choosing a sufficiently small $\theta > 0$
and making  $\varepsilon_0 > 0$ smaller if necessary, we achieve that 
\begin{equation}\label{rela1}
\inf_{\bar{D}} w^\la(\cdot, t) > r_1 \qquad (t \in [\tau, \tau +
8\theta], \la \in [\la_0 - \varepsilon_0, \la_0])\,. 
\end{equation}
Having fixed $d$, $\rho$, $\theta$, let $p$ 
and $\tilde{\mu}$  be as in Theorem \ref{imp}.

By \eqref{pos0} and the H\"older estimate \eqref{eq:hol1U}, one has 
\begin{align*}
w^\la(x, \tau) &= w^{\la_0}(x, \tau) + u(P_\la x, t) - u(P_{\la_0} x, t) \geq 
- C |P_\la x - P_{\la_0} x|^{\alpha} \\
&= -2C|\la_0 - \la|^\alpha \qquad (x \in \Omega_{\la_0})
\end{align*}
and 
\begin{equation*}
w^\la(x, \tau) \geq -C |P_{\la}x - x|^\alpha \geq -2C|\la_0 - \la|^\alpha \qquad (x \in \Om_{\la}\setminus \Omega_{\la_0})\,.
\end{equation*}
Thus decreasing $\varepsilon_0>0$ further if
needed, one achieves 
\begin{equation}\label{rela2}
\|(w^\la)^-(\cdot, \tau)\|_{L^\infty(\Om_\la)} < \tilde{\mu} r_1
\qquad (\la \in [\la_0 - \varepsilon_0, \la_0]) \,.
\end{equation}
Relations \eqref{rela1} and \eqref{rela2} show that the second
inequality in  \eqref{amu} is satisfied with $v= w^\la$ and
$G=\Om_\la$. Consequently, \eqref{pvd} and \eqref{uvd} hold true. By
\eqref{uvd} we obtain 
\begin{multline}\label{nnw}
\|(w^\la)^-(\cdot, t)\|_{L^\infty(\Om_\la)}
\leq 2 e^{-(t - \tau)} \|(w^\la)^-(\cdot, \tau)\|_{L^\infty (G)}  \\\qquad
 (t > \tau, \ \la \in (\la_0 - \varepsilon_0, \la_0])\,,
\end{multline}
and therefore, by passing to the limit as $t \to \infty$,
$V_\la z \geq 0$ in $\Om_{\la}$ for each  
$\la \in (\la_0 - \varepsilon_0, \la_0]$ and each $z \in \omega(U)$.
Combining this result with Lemma \ref{l:dec}, we conclude that 
$\La(z)\le \la_0 - \varepsilon_0$ for each $z \in
\omega(U)$. Statement (i) is proved.

Next we prove statement (ii). 
Choose a sequence $(t_n)_{n \in \N}$
such that $t_n \to \infty$ and $U(\cdot, -t_n)  \to z$ as $n \to \infty$.
Denote 
\begin{equation*}
 r_1 := \frac{1}{2} \inf_{\bar{D}} V_{\la_0}z > 0 \,.
\end{equation*}
Then, by \eqref{eq:hol1U}, there are $\theta$ and $n_0$ such that, possibly
after  $\varepsilon_0 > 0$ is made smaller, one has 
\begin{equation*}
\inf_{\bar{D}} w^\la(\cdot, t) \geq r_1 \qquad (t \in [-t_n, -t_n + 8\theta), \la \in (\la_0 - \varepsilon_0, \la_0], n \geq n_0)\,.
\end{equation*}
Having fixed $\theta$ (and  $d$, $\rho$), let $p$ 
and $\tilde{\mu}$  be as in Theorem \ref{imp}.

We can now argue as  in the
 proof of statement (i), taking $\tau=-t_n$, with  $n\ge n_0$,
 in the arguments following \eqref{rela1}. Specifically,
we first  make $\varepsilon_0 > 0$ yet smaller (independently of $n$) to
 achieve that \eqref{rela2} holds. This implies that \eqref{nnw} holds
 with  $\tau=-t_n$.     
Since $w^\la$ is bounded, passing to the limit as $n \to \infty$, we obtain 
$w^\la(\cdot, t) \geq 0$ in $\Om_{\la}$ for each $t\in\R$ and 
 $\la \in (\la_0 - \varepsilon_0, \la_0]$. In particular, 
 $w^{\la_0}(\cdot, t) \geq 0$ in $\Om_{\la_0}$. Clearly,  in 
 view of the assumption of (ii), we can
 replace $\la_0$ with any other value  $\tilde \la_0\in (\la_0,\ell)$,
 hence   $w^\la(\cdot, t) \geq 0$ in $\Om_{\la}$ for all 
    $t\in\R$ and $\la \in (\la_0 - \varepsilon_0, \ell]$.
This proves that $\La(U(\cdot,t))\le \la_0 - \varepsilon_0$
for each $t\in\R$. 
\end{proof}

\subsection{Two symmetry results for nonautonomous equations}
\label{symm-nonaut}
In this subsection we state two symmetry results from \cite{P:est,
  P:symm-survey}, as they apply to entire solutions of \eqref{tdc}.
The results in \cite{P:est,  P:symm-survey} 
concern  fully nonlinear equations of 
which \eqref{tdc} (hence also \eqref{lim}) is a special case. 
We will use these results in the proof of Theorem \ref{thm2}.

We assume that $g$ is a function satisfying conditions (G1)-(G3). 

\begin{Theorem}\label{thmest}
Let $U\in \cA^*$ and
\begin{equation*}
  \la^*=\sup \{\La(z): z\in    \om(U)\}.
\end{equation*}
Then $\la^* \in [0,\ell)$ and 
for each $z\in \om (U)$ one has $V_{\la^*}z\equiv 0$ on some
connected component of $\Om_{\la^*}$.  If
$\om(u)$ contains a function $z_0$ such that $z_0>0$ in $\Om$, then 
$\la^* =0$.
\end{Theorem}

 The  existence of 
$\la^* \in [0,\ell)$ satisfying the first conclusion 
is proved in \cite[Theorem 2.4]{P:est}. As shown there
 (see the beginning of the proof of Lemma~4.2 in
\cite{P:est}), 
$\la^*$ is given by
\begin{equation*}
  \la^*=\inf\{\mu\ge 0: \  V_\la(z)\ge 0\text{ in $\Om_\la$ for all } 
\la\in (\mu,\ell) \textrm{ and } z\in    \om(U)\}
\end{equation*}
which is the same as 
\begin{equation*}
   \la^*=\inf\{\mu\ge 0: \  \La(z)\le \mu \ \ (z\in    \om(U))\}=\sup \{\La(z): z\in    \om(U)\}.
\end{equation*} 
The fact that
$\la^* =0$ if $\om(u)$ contains a strictly positive function is 
proved in \cite[Theorem 2.2]{P:est}. In this case, 
each $z\in \om (U)$ is even in $x_1$ and monotone nonincreasing
in  $x_1>0$.

\begin{Theorem}
  \label{thmsurvey} Let $U\in \cA^*$. 
If $\al(U)$ contains a function $z_0$ such that $z_0>0$ in $\Om$, then 
for each $t\in \R$ one has $\La(U(\codt,t))=0$  and  the function
$U(\codt,t)$ is even in $x_1$ and decreasing in $x_1$.
\end{Theorem}
This is stated in \cite[Theorem 3.4]{P:symm-survey}. 
As indicated in 
\cite[Sections 4.2,~6]{P:symm-survey}, the theorem
can be proved by similar techniques as Theorem \ref{thmest}
if the general scheme of the proof is suitably adjusted to deal with
the symmetry at all times rather than with the asymptotic symmetry as
$t\to\infty$.  For the reader's convenience, in Appendix~A
we give an alternative proof based on the 
properties of the functional $\La$
established in the previous subsection and a generalized Harnack
inequality from \cite[Theorem 2.2]{P:est}.
Under stronger  hypotheses requiring in particular the
strict positivity of all elements of $\al(u)$ and a sign condition on
the nonlinearity, the symmetry result is proved in 
\cite{Babin:symm1, Babin-S}.

\section{Proof of Theorem \ref{thm2}}\label{proof}
Throughout the section we assume the
hypotheses (D1), (D2), (A), (F1), and (F2) 
to be satisfied.

We start with two results concerning  equilibria.
The proof of the following lemma can be found in  
\cite{P:symm-ell}.   

\begin{lemma}
  \label{le:char} 
Let  $z\in E$. Then the following statements are valid.
  \begin{itemize}
  \item[(i)] $z\in E_+$ if and only if $z\not \equiv 0$ and $z$
    vanishes somewhere in $\Om$.    
 \item[(ii)] If $z\not\equiv 0$, then $z$ does not vanish on any open subset of $\Om$ and 
   $\la=\La(z)$ is the maximal
   number in $(0,\ell)$ with the property that the function $V_\la z$ vanishes
   identically in a connected component of $\Om_\la$. 
  \end{itemize}
\end{lemma}

\begin{lemma}\label{efi} The set $E_+$  has only finitely many elements.
Moreover, for each $z\in E_+$ one has $\dist_{C_0(\bar{\Om})}(z, E_0)>0$.
\end{lemma}

\begin{proof}
  The second statement follows from the fact that each element of
$E_0$ is monotone in $x_1$ in the set $\Om_0$, whereas $z\in E_+$ is
obviously not monotone. 

The first statement is a result of   
\cite[Theorem 2.1]{P:symm-mult}; however, a remark  on the applicability of
\cite{P:symm-mult} is necessary. 
In \cite{P:symm-mult} fully nonlinear equations were considered. To
guarantee that a linearization of the equation has   
Lipschitz continuous coefficients in the principal part, a slightly
higher regularity of $U$ had to be required.  In the present case, the
leading part of the equation is always 
given by the Laplacian, hence 
the extra regularity assumption is not needed.
\end{proof}

We next recall an invariance property of the $\om$ and $\al$-limit sets.

\begin{lemma}\label{eqc} If $U \in \mathcal{A}$
and $z \in\omega(U)$ (or $z \in \alpha (U)$), 
then there is $Z \in \mathcal{A}$ with $z = Z(\cdot, 0)$ and 
$Z(\cdot, t) \in \omega (U)$ (or, respectively, $Z(\cdot, t) \in \alpha (U)$) 
for any $t \in \R$.
\end{lemma}

\begin{proof} 
This is quite a standard result, however, since we do not assume any
smoothness of $\partial\Om$, we need to deal with some regularity issues. 
We carry out the details 
for $\om(U)$,  $\al(U)$ can be dealt with similarly.

We will employ the global H\"older estimate
\eqref{eq:hol1U} as well as the following
interior $L^p$--estimate. For any $p\in (1,\infty)$, $T\in (0,\infty)$,
and any domain $\tilde \Om$ whose closure is contained in $\Om$ one has
\begin{equation}\label{lpest}
\sup_{s \in \R} \|U\|_{W^{2,1}_p(\tilde \Om\times (s-T,s+T))} <\infty\,.     
\end{equation}
To prove  \eqref{lpest}, rewrite the equation for $U$ as follows:
\begin{align}
  U_t&=\De U+f(x,U,\nabla U)-f(x,0,0)+f(x,0,0)\notag\\
&=L(x,t) U+f(x,0,0),\label{eq:forlp}
\end{align}
where $L \in \cE(  \beta_0, \Om \times \R)$ and $\beta_0$ is as in
(F1). The function $f(x,0,0)$ is bounded by (F1). 
Applying to this
 linear nonhomogeneous equation
the interior $L^p$--estimates  (see 
\cite{Ladyzhenskaya-S-U, Lieberman:bk}), 
one obtains \eqref{lpest}.

 Let now $z\in \om(U)$. 
There is a sequence $s_m \to \infty$ such that 
$U(\cdot, s_m) \to z$ in $C_0(\bar{\Om})$. Using  \eqref{eq:hol1U} 
 and a diagonalization procedure, replacing $s_m$ 
by a subsequence if necessary, one
shows that the limit 
\begin{equation}
  \label{eq:1lim}
  Z:=\lim_{m\to\infty} U(\cdot, \cdot + s_m)
\end{equation}
exists, pointwise and  uniformly on the compact subsets of
$\bar\Om\times \R$. Clearly, $Z(\codt,0)=z$ and  
$Z(\codt,t)\in \om(U)$ for each $t$. 
Using the reflexivity of the Sobolev spaces
$W^{2,1}_p(\tilde \Om\times (s-T,s+T))$, 
$p\in (1,\infty)$, and their
continuous imbedding in H\"older spaces  \cite{Ladyzhenskaya-S-U}, one 
further obtains, replacing $s_m$ 
by a subsequence again, that the limit in \eqref{eq:1lim} takes place
in $C^{1+\be,\be/2}_{\textrm{loc}}(\Om\times \R)$ for each $\be\in (0,1)$,
as well as weakly in $W^{2,1}_p(\tilde \Om\times (s-T,s+T))$ for each
$p\in (1,\infty)$, $T>0$, and $\tilde \Om$ as above.
In particular, $Z\in C^{1+\be,\be/2}_{\textrm{loc}}(\Om\times \R)$ and
\begin{equation*}
 f(\cdot, U(\cdot, \cdot + s_m), \nabla U(\cdot, \cdot + s_m))\to 
f(\cdot, Z, \nabla Z)
\end{equation*}
locally uniformly in $\Om\times \R$. Using test
functions and passing to the limit in the equation for $U$, we obtain
that $Z$ satisfies in the generalized sense the  equation
\begin{equation}
  \label{eq:nonhheat}
  Z_t=\De Z+\Phi (x,t),\quad x\in \Om,\,t\in\R,
\end{equation}
where $\Phi (x,t)=f(x, Z(x,t), \nabla Z(x,t))$. By \eqref{eq:3}, this
function is H\"older continuous, hence by Schauder theory $Z$ is a
classical solution of \eqref{eq:nonhheat}. 
\end{proof}

Note that $0\in \cA$ if an only if $0$ is an equilibrium of
\eqref{lim}, that is, if and only if $f(\codt,0,0)\equiv
0$. Below, when writing $U\in \cA\setminus\{0\}$ we mean that 
$U\in \cA$ and $U\not\equiv 0$ (in case $0\in E$). The next lemma
shows that this is the same as saying that $U\in \cA$ and 
$U(\codt,t)\not\equiv 0$ for any $t\in\R$ on any open subset of $\Om$.

\begin{lemma}
  \label{claim4}
Let $U \in \mathcal{A}$. If $U(\codt,\tau)\equiv 0$ on an open subset
$G$ of $\Om$ for some 
$\tau\in \R$, then  $U \equiv 0$.
\end{lemma}
\begin{proof}
Since $U \geq 0$, $0$ is a local minimum of $U$. 
Thus, $U_t(\cdot, \tau) \equiv 0$ on $G$. Of course, 
one also has $|\nabla U(\cdot, \tau)| \equiv \Delta U(\cdot, \tau)
\equiv 0$ in $G$.  
From \eqref{lim}, 
we obtain $f(x, 0, 0) = 0$ for each $x \in G$, and therefore $U$ solves
\begin{equation*}
U_t = \Delta U + f(x, U, \nabla U) - f(x, 0, 0) = L(x, t)U, \qquad (x,
t) \in G \times \R\,, 
\end{equation*}
where $L \in \mathcal{E}(\beta_0, G \times \R)$. Since $U \geq 0$ and
$U (\cdot, \tau) \equiv 0$ on $G$,  
the strong maximum principle yields $U \equiv 0$ on $G \times
(-\infty, \tau)$. For any $\delta \in (0, 1)$  
define $w(x, t) := U(x, t + \delta) - U(x, t)$. Then $w$ solves 
\begin{equation*}
\begin{aligned}
w_t &= L(x, t) w, &\qquad &(x, t) \in \Om \times \R \,, \\
  w &= 0,  &\qquad &(x, t) \in \partial \Om  \times \R \,,
\end{aligned}
\end{equation*}
and $w \equiv 0$ on $G \times (-\infty, \tau - 1)$. Therefore, by 
the weak unique continuation theorem (see e.g. \cite{Alessandrini:V-uni:cont})
$w \equiv 0$ on $\Om \times (-\infty, \tau - 1)$. 
Consequently, by the uniqueness for the Dirichlet 
initial-boundary value problem (which follows from the maximum
principle), $w \equiv 0$ on $\Om \times (-\infty, \infty)$. 
Hence $U(\cdot, \cdot) \equiv U(\cdot, \cdot + \delta)$ for each 
$\delta \in (0, 1)$, and therefore $U$ is
an equilibrium. Since $U \equiv 0$ on $G$,  Lemma \ref{le:char}
 implies that $U \equiv 0$ in $\Om$.   
\end{proof}

The following lemma  
is crucial for our further arguments. 
It shows that entire solutions with certain
additional properties have to be equilibria.

\begin{lemma} 
\label{cor:equil}
Let $U\in \cA\setminus \{0\}$, $\la_0 \in (0,\ell)$,  
and let $I$ be an open interval.
Assume that for each $t\in I$, one has 
$V_{\la_0} (U(\cdot, t)) \equiv 0$  on 
 some connected component $D^t$  of $\Om_{\la_0}$. Then
$U\in E_+$. 
\end{lemma}

\begin{proof} Recall that by hypothesis (D2), $\Om_{\la_0}$ has only
  finitely many connected components. 
Since the function $V_{\la_0}U $ is continuous, shrinking the 
 interval $I$  if necessary, we may assume that  the connected
component $D=D^t$ is independent of $t$. 
Since $\la_0 > 0$, $\mathcal{M} \neq \emptyset$, where 
\begin{align*}
\mathcal{M} &:= P_{\la_0} (\partial D \cap \partial \Om) \cap \Om \,. 
\end{align*}
Now by standard interior Schauder estimates,
$U \in C_{\textrm{loc}}^{2 + \alpha_0, 1 + \alpha_0/2}(\Om \times I)$, where $\alpha_0 > 0$ is as in (F1). 
Next, we see that 
\begin{equation}\label{grad0}
U(x, t) =U_t(x,t)= |\nabla U(x, t)| = 0 \qquad ((x, t) \in  \mathcal{M} \times I) \,.
\end{equation}
Indeed, we have
$U = 0$ on $\mathcal{M} \times I$,
since $V_{\la_0} U \equiv 0$ on $D$ and $U$ satisfies the
Dirichlet boundary condition. Since 
 $U \geq 0$ in $\Om \times I$, any point in 
$\mathcal{M} \times I$ is a local minimizer of $U$. Thus
$U_t\equiv |\nabla U| \equiv 0$  on $\mathcal{M} \times I$.

We next prove the following claim.

\vspace{.2cm} 
\noindent 
\textit{Claim.} $\mathcal{M}$ contains an  $(N-1)$-dimensional
$C^{1+\alpha_0}$ manifold  
$\Upsilon$.
 \vspace{.2cm} 

\noindent
Let ${\mathcal{P}} : \R^N \mapsto H_0$, be the orthogonal projection to $H_0$.
Clearly, ${\mathcal{P}} (D) = {\mathcal{P}} (\mathcal{M})$ 
and ${\mathcal{P}} (D)$ has nonempty relative interior in $H_0$.

We show that there is $(x_0, t_0) \in \mathcal{M} \times I$ such that
$D^2 U(x_0, t_0) \neq 0$.    
Otherwise, $\Delta U= 0$ on $\mathcal{M} \times I$,
and combined with \eqref{grad0} this gives $f(x, 0, 0) = 0$ 
for each $x \in \mathcal{M}$. 
Since $f$ is independent of $x_1$, one has $f(\cdot, 0, 0) \equiv 0$
on the open cylinder $\R \times {\mathcal{P}} (D)\subset \R^N$. 
Fix $x_0 \in \mathcal{M}$ and $R > 0$ 
such that $B(x_0, R) \subset (\R \times {\mathcal{P}} (D))\cap\Om$.   
Then $U$ satisfies 
\begin{equation*}
U_t = \Delta U + f(x, U, \nabla U) - f(x, 0, 0) = L(x, t)U, \quad
(x, t) \in B(x_0, R) \times \R \,, 
\end{equation*}
where $L \in \mathcal{E}(\beta_0, B(x_0, R) \times \R)$, 
and $U \geq 0$, $U(x_0, t) = 0$ for $t \in I$. 
Thus by the strong maximum principle $U \equiv 0$ 
on $B(x_0, R)\times I$, and consequently, by 
Lemma \ref{claim4}, $U \equiv 0$ everywhere, a contradiction.  

Hence, we have showed that there are 
$(x_0, t_0) \in \mathcal{M} \times I$
and $k, l \in \{1, \cdots, N\}$  
such that $\partial_{x_k x_l} U (x_0, t_0) \neq 0$. Set $v :=
(U)_{x_k}$. Then one has
$
v \in C_{\textrm{loc}}^{1 + \alpha_0, \alpha_0/2}(\Om \times I)$, $|\nabla v(x_0,
t_0)| \neq 0,$ 
 and, by \eqref{grad0}, $v \equiv 0$ on $\mathcal{M}
\times I$. Denote $\mathcal{Z}:= \{x \in \Om : v(x, t_0) = 0\}$.
By the implicit function theorem, there is $r > 0$ and an  
$(N - 1)$-dimensional $C^{1+ \alpha_0}$ manifold $\Upsilon$ such that  
$\mathcal{Z} \cap B(x_0, r) = \Upsilon \cap B(x_0, r)$. We can also
assume that $\Upsilon$ divides  $B(x_0, r)$ into exactly 2 connected
components.  By \eqref{grad0},  
$\mathcal{M} \cap B(x_0, r) \subset \Upsilon \cap B(x_0, r)$.
We finish the proof of the Claim 
by showing that $\mathcal{M} \cap B(x_0, r) = \Upsilon \cap B(x_0,
r)$. Indeed, if it is not true, then the set
$B(x_0, r) \setminus \mathcal{M}$ is connected.
Consequently,
$P_{\la_0}(B(x_0, r) \setminus \mathcal{M})$
is connected as well.  Clearly this open connected set
contains points of $\Om$, and therefore it cannot contain any points of
$\R^N\setminus\bar\Om$. Thus $P_{\la_0}(B(x_0, r) \setminus
\mathcal{M}) \subset \Om$. 
Since $\partial \Om \subset P_{\la_0}(B(x_0, r) \cap \mathcal{M})$
is a subset of the $(N - 1)$-dimensional manifold
$P_{\la_0}(\Upsilon)$, we obtain a contradiction 
to condition (A). The Claim is proved.

To continue, we fix a ball  $G \subset \Omega$ 
such that $\Upsilon$ divides $G$ into 
two connected components $G_+$ and $G_-$. 
Denote by $\tau$, $T$ the boundary points of the interval $I$: 
$I= (\tau, T)$, and  set
$I_0 := (\tau, (T + \tau)/2 )$. Fix any $\delta \in (0, (T - \tau)/2)$
and denote $w(x, t) := U(x, t + \delta) - U(x, t)$ 
for $(x, t) \in \Om \times \R$. Notice that $t + \delta \in I$
whenever $t \in I_0$. Then
$w \in C^{2 + \alpha_0, 1 + \alpha_0/2}(\Om  \times \R)$ satisfies 
\begin{equation*}
\begin{aligned}
w_t = L(x, t) w, \qquad (x, t) \in \Om \times \R \,, 
\end{aligned}
\end{equation*}
where $L \in \cE(  \beta_0, \Om \times \R)$. We next consider a new
operator  $L^\ast\in \cE(  \beta_0, \Om \times \R)$ defined by
 $L^\ast = L$ on $G_+ \times I_0$ and $L^\ast = \Delta$ on 
$G_- \times I_0$ (the lower order coefficients are equal to 0). 
Also define $w^\ast$ such that $w^\ast = w$ on $G_+ \times I_0$ and
$w^\ast \equiv 0$ on $G_- \times I_0$. Since  
$w$ is a  solution on $G$ and $w = |\nabla w| = 0$ on $\Upsilon$,
an integration by parts shows that 
$w^\ast$ is a weak solution of 
\begin{equation*}
\begin{aligned}
w_t^\ast = L^\ast(x, t) w^\ast, \qquad (x, t) \in G \times I_0 \,.
\end{aligned}
\end{equation*}
Since $w^\ast \equiv 0$ on the open set $G_-  \times I_0$, 
the unique continuation principle
 \cite{Alessandrini:V-uni:cont} yields
$w^\ast \equiv 0$ on $G \times I_0$. Thus $w = w^\ast = 0$ on $G_+
\times I_0$ and the unique  
continuation for $w$ yields $w \equiv 0$ on $\Omega\times
I_0$. Consequently, $ U \equiv U(\codt, \cdot+\delta)$ on  
$\Om \times I_0$ for any sufficiently small $\delta$. 
The uniqueness for the initial value problem implies that 
$ U \equiv U(\codt, \cdot+\delta)$ on  
$\Om \times (\tau,\infty)$ and the backward uniqueness for parabolic
equations   (see Remark \ref{rk:bu} below) then 
gives $ U \equiv U(\codt, \cdot+\delta)$ on $\Om \times\R$, hence
$U \in E$. Since $U\not  \equiv 0$,  
the assumption that $V_{\la_0} U \equiv 0$ on $D$ implies that $U$ has
nontrivial nodal set,  
thus $U\in E_+$ (see Lemma \ref{le:char}). 
\end{proof}

\begin{remark}\label{rk:bu}
\textnormal{
This remark is to justify the use of backward uniqueness 
in the previous proof and in an argument given in the next section.
The  backward uniqueness theorems from 
\cite{Ghidalia-back:uni, Temam:bk}, for example,  
apply to the difference of any two solutions of \eqref{lim}
provided the following  statement holds. 
Given any solution $U$ of   
\eqref{lim}, the  function
$\tilde U:t\mapsto U(\codt,t)$ satisfies
\begin{equation}
  \label{eq:regt}
  \tilde U \in C(\R, H^1_0({\Om})) \cap L^2_{\textrm{loc}}(\R, D(\Delta)).
\end{equation}
Here $D(\Delta)$ is the domain of the $L^2(\Om)$-realization of the
Laplace operator with Dirichlet boundary conditions:
\begin{equation*}
  D(\Delta)=\{\varphi \in H^1_0(\Om): \De \varphi \in L^2(\Om)\}, 
\end{equation*}
where $\De$ is viewed as an isomorphism of $H^1_0(\Om)$ onto
$H^{-1}(\Om)$ (one has $D(\Delta)=H^2(\Om)\cap H^1_0(\Om)$ if $\Om$ is
smooth).  The space $D(\Delta)$ is equipped with the usual graph norm.
For smooth domains $\Om$, \eqref{eq:regt} is well known. 
In the general case, \eqref{eq:regt} can be established by a rather
standard approximation procedure. For the reader's convenience, we
give the details in Appendix B.
}
\end{remark}

\begin{lemma}\label{l:ref-2}
Let  $U \in \mathcal{A}$ and  $\la_0 := \La(U(\cdot, \tau)) > 0$ for
some $\tau\in \R$. Then either $U\in E_+$ or there exists
$\varepsilon > 0$ such that 
\begin{equation}
  \label{eq:4}
  \La(z) \leq \la_0 - \varepsilon\qquad(z \in \omega(U)).
\end{equation}
\end{lemma}
\begin{proof}  
By Lemma \ref{l:dec}, 
$\la_0= \La(U(\cdot, \tau))\ge  \La(U(\cdot, \tau_0))$ for each
$\tau_0\ge \tau$. By Lemma \ref{l:ref-2-moved}, relation \eqref{eq:4} holds
for some $\varep>0$, provided there is $\tau_0\ge \tau$ such that
$V_{\la_0}(U(\cdot, \tau_0))\not\equiv 0$ on any connected component   
of $\Om_{\la_0}$. On the other hand, if there is no such $\tau_0\ge
\tau$, then Lemma \ref{cor:equil} applies and we conclude that
 $U\in E_+$.
\end{proof}

\begin{lemma}\label{ols}
If $U \in \mathcal{A} \setminus E_{+}$, $\tau \in \R$,
and $\la_0:=\La(U(\cdot, \tau))$,
then either $\omega(U) = \{z_0\}$ for some $z_0 \in E_+$ with
$\La(z_0) < \la_0$, or 
\begin{equation}
  \label{eq:st}
  \text{$\La(z) = 0$ for each  $z \in \omega(U)$.}
\end{equation}

\end{lemma}

\begin{proof}
By Lemma \ref{l:dec}, \eqref{eq:st} holds if $\la_0=0$. 

Assume
that $\la_0>0$ and \eqref{eq:st} does not hold, that is,
there is $z_0 \in \omega(U)$ with $\La(z_0) > 0$. 
Let  $\la^*\ge 0$ be as in Theorem \ref{thmest}. 
Then,  by Lemma \ref{l:ref-2},
 $\la^*< \la_0$ and, obviously,  $\la^* \ge \La(z_0)>0$. 
Once we know that $\la^*>0$, we can apply Lemma \ref{cor:equil} to
each nonzero entire solution in $\om(U)$ (cp. Lemma \ref{eqc}).
Indeed, any such entire solution $Z$ satisfies $Z(\cdot, t) \in \omega(U)$, 
and therefore for each $t \in \R$ one has $V_{\la^*}Z(\cdot, t) \equiv 0$
on a connected component of $\Om_{\la^*}$, as stated in Theorem
\ref{thmest}.  By Lemma \ref{cor:equil}, $\om(U)\setminus \{0\}\subset
E_+$.  Hence, by Lemma \ref{efi}, $\om(U)$ is a finite set and, as it
is connected, $\om(U)=\{z_0\}$. 
Since  $\La(z_0)>0$, we have $z_0\in E_+$ and the relations
$\la_0>\la^* \ge \La(z_0)$ complete the proof. 
\end{proof}

The above results are mainly concerned with $\om(U)$.
Next we intend to consider $\al(U)$. 
For that the following  lemma
will be useful.

\begin{lemma}\label{psu}
Let $Z \in \mathcal{A}
 \setminus \{0\}$ and $\tau\in \R$.
Set $\la_0 := \La(Z(\cdot,\tau ))$. 
Then for each $\la \in(\la_0, \ell)$ 
one has $V_{\la} Z > 0$ in $\Om_{\la} \times (\tau, \infty)$.
Moreover if $\la_0 > 0$, then either $V_{\la_0} Z > 0$ in $\Om_{\la_0} \times (\tau, \infty)$ or $Z \in E_+$. 
\end{lemma}
\begin{proof}
For $\la\ge \la_0$  let
$w^{\la}$ be as in \eqref{dfw}; it 
satisfies \eqref{ref},  \eqref{bcd}, and
\eqref{pos}. By the maximum principle, 
 either 
 \begin{equation}
   \label{eq:2}
   \text{$w^{\la} > 0$ in $\Om_{\la} \times (\tau, \infty)$}
 \end{equation}
 or there
 are $\ep>0$ and a connected component $D_{\la}$ of $\Om_{\la}$ such
 that 
$w^{\la}\equiv 0$ in $D_{\la} \times (\tau,\tau +\ep)$.  
The second possibility cannot hold if  $\la>\la_0$. For if it did, then,  by 
Lemma \ref{cor:equil}, $Z\in E_+$, and we would have a  contradiction
to Lemma \ref{le:char}(ii). Hence, \eqref{eq:2} holds for
$\la>\la_0$. If $\la=\la_0$ and $\la_0>0$, then 
either \eqref{eq:2} holds or Lemma \ref{cor:equil} implies $Z\in E_+$. 
The lemma is proved.
\end{proof}


\begin{lemma}\label{prs}
Let $U \in \mathcal{A}$ and assume that
$\mu_0 := \La (U(\cdot, \tau)) > 0$ for some $\tau \in \R$. Then
$\als(U) = \{z\}$ for some $z\in E$,
and either $z \equiv 0$,  or 
$z \in E_+$ and  $\La(z) \geq \mu_0$. 
\end{lemma}

\begin{proof}
By Lemma \ref{l:dec},  $\sigma :=\lim_{\tau\to-\infty}\La (U(\cdot,
\tau))\ge \mu_0>0$ and 
\begin{equation} \label{eq:7}
  \sigma\ge \La(z)\quad(z\in \al(U)).
\end{equation}
We claim  that
\begin{equation}
  \label{eq:8}
   \La(z)= \sigma\quad(z\in \al(U)\setminus\{0\}).
\end{equation}
To prove this, take any $z_0\in \al(U)\setminus \{0\}$. 
There is $Z\in \cA$ with $Z(\codt,0)=z_0$ 
and $Z(\cdot,t)\in \al(U)$ for all $t\in\R$ 
(cp. Lemma \ref{eqc}). Then $\La(Z(\cdot,1))\le \La(z_0)$,
by Lemma \ref{l:dec}. By Lemma \ref{psu},  we have
$V_{\la} Z(\cdot,1)>0$ in $\Om_{\la}$ for any 
$\la \in (\La(z_0),\ell)$.
Applying Lemma \ref{l:ref-2-moved}(ii) with $z=Z(\cdot,1)$,
we obtain $\La(U(\codt,t))<\la$  for each $t\in \R$, 
hence $\sigma\le \la$. Since $\la \in (\La(z_0),\ell)$ was
arbitrary, it follows that $\sigma\le \La(z_0)$. Combined with
\eqref{eq:7}, this gives $\La(z_0)= \sigma$. 
Hence  \eqref{eq:8} is proved. 

We next prove that $\al(U)\setminus \{0\}\subset E_+$. Take again 
any $z_0\in \al(U)\setminus \{0\}$
and let $Z$ have the same meaning as in the previous paragraph. 
Observe  that $Z(\cdot,1)\not\equiv  0$ (otherwise,
$0\equiv Z(\cdot,0)\equiv  z_0$ by Lemma \ref{claim4}).  Thus, 
 \eqref{eq:8}  gives $\La(Z(\cdot,1))=\sigma>0$. 
Assume $Z(\cdot,1)\not\in E_+$ and set $\la_0 := \La(z_0)$.   
By Lemma \ref{psu}, for each $\la \in [\la_0, \ell)$ one has
$V_{\la}Z(\cdot, 1) > 0$  
in $\Om_{\la}$, and consequently
we can apply Lemma \ref{l:ref-2-moved}(ii) with $z=Z(\cdot,1)$. This yields $\varep_0>0$ such that
$\La(U(\codt,t))\le \la_0-\varep_0$  for each $t\in \R$ and therefore
$\sigma\le \La(z_0)-\varep_0$, in  contradiction to
\eqref{eq:8}. This contradiction shows that 
$Z(\cdot,1)\in E_+$. Thus  $z_0=  Z(\cdot,0)= Z(\cdot,1)\in E_+$,
as desired. 

Once we know that $\al(U)\setminus \{0\}\subset E_+$, Lemma
\ref{efi} implies that $\al(U)$ is a finite set. As it is connected, 
$\al(U)$ consists of a single equilibrium $z$ and either $z\in E_+$ or
$z\equiv 0$. 
\end{proof}

The next lemma treats the case  $\al(U)=\{0\}$.

\begin{lemma}\label{nul}
If $U\in \cA$ and $\al(U)=\{0\}$, then
$\La (U(\cdot, t)) = 0$ for each $t\in \R$.
\end{lemma}
We postpone the proof of this lemma until the next subsection. 

We are
now ready to complete the proof  Theorem \ref{thm2}.

\begin{proof}[Proof of Theorem \ref{thm2}] 
It is obvious that at most one of the statements (i)-(iv) of 
Theorem \ref{thm2} can hold.

If  $\al(U)=\{0\}$, then Lemma \ref{nul} says that statement (i) holds. 

Assume now that   $\al(U)\ne\{0\}$ and none of the
  statements (i), (ii) holds, that is,
 $U \not \in E_+$ and there is $\tau \in \R$ 
such that $\mu_0=\La(U(\cdot, \tau))  > 0$. 
Then,  Lemmas \ref{ols}, \ref{prs} imply that one of the statements 
(iii), (iv) holds.  

The conclusion concerning the case $f(\cdot, 0, 0) \ge 0$ follows from
the comparison principle, as already explained in Remark \ref{rmtothm2}(c).
\end{proof}

\subsection{Proof of Lemma \ref{nul}}
\label{sec:nul}
Throughout this subsection we assume that $U\in \cA$ and 
$\al(U)=\{0\}$. 
Lemmas  \ref{eqc} and \ref{claim4},
imply that $0\in E$ and  $f(\cdot, 0, 0) \equiv 0$.
The conclusion of Lemma \ref{nul} trivially holds true 
if  $U \equiv 0$, thus in the following we assume 
$U \not \equiv 0$. By Lemma  \ref{claim4}, $U(\cdot,t)\not\equiv 0$
for any $t\in \R$, hence, by the strong comparison principle,
 $U(\cdot,t)> 0$ in $\Om$ for each  $t\in \R$.

Now using $f(\cdot, 0, 0) \equiv 0$, we have  
$U_t=\De U+f(x,t,U,\grad U)-f(x,t,0,0)$. Therefore, by the
Hadamard formula,
$U$ is a bounded positive entire solution of  
a linear problem
\begin{equation}\label{eps}
\begin{aligned}
v_t &= \Delta v + L(x, t)v,& \quad &(x, t) \in \Om\times\R \,,\\
v &= 0,& \quad &(x, t) \in \partial \Om\times\R \,,
\end{aligned}
\end{equation}
where $L \in \mathcal{E}( \be_0, \Om \times \R)$ and $\be_0$ is as in
(F1). We use this
observation below to control the decay of $U(\cdot,t)$, as
$t\to\infty$. Then we find a suitable transformation  of $U$
which is uniformly positive. This will allow us to 
apply Theorem \ref{thmsurvey} to conclude that 
$\LA(U(\codt,t))=0$ for all $t$.

Below, $C^*, C_1,C_2,...$ denote positive constants independent of $t$
and $x$.

\begin{lemma}
  \label{le:est} One has
\begin{align}\label{dno}
0 < C_1 \leq \frac{\|U(\cdot, t + \tau)\|_{L^\infty(\Om)}}{\|U(\cdot, t)\|_{L^\infty(\Om)}} &\leq C_2 < \infty \qquad 
(t \in \R, \tau \in [0, 1]) \,
\\
\intertext{and} 
    \label{eq:harn}
    \inf_{t\in \R}\frac{U(x, t)}{\|U(\cdot,
      t)\|_{L^\infty(\Om)}}&>0\qquad(x\in \Om). 
  \end{align}
\end{lemma}

\begin{proof}
As remarked above, 
$U$ is a positive bounded  solution of  a  linear problem \eqref{eps}
with
 $L \in \mathcal{E}( \be_0, \Om \times \R)$. 
  By \cite[Theorem 5.5]{Huska:nondiv}, 
there exists a positive solution $\phi$ of \eqref{eps}
satisfying the following two conditions
\begin{align}\label{pso}
0 < C_1 \leq \frac{\|\phi(\cdot, t + \tau)\|_{L^\infty(\Om)}}{\|\phi(\cdot, t)\|_{L^\infty(\Om)}} &\leq C_2 < \infty \qquad 
(t \in \R, \tau \in [0, 1]) \,,\\
 \inf_{t\in \R}\frac{\phi(x, t)}{\|\phi(\cdot,
      t)\|_{L^\infty(\Om)}}&>0\qquad(x\in \Om).\notag
\end{align} 
Note, in particular that \eqref{pso} implies 
\begin{equation}\label{eso}
\|\phi(\cdot, t)\|_{L^\infty(\Om)} \leq C_3 e^{\gamma |t|} \qquad (t \in \R)
\end{equation}
for some $C_3, \gamma > 0$. By 
\cite[Proposition 2.5]{Huska:nondiv},  
the positive solution satisfying \eqref{eso} is unique
up to scalar multiples. Since $U$ is bounded and positive, $U =
c\phi$ for some $c > 0$. 
This implies \eqref{dno}, \eqref{eq:harn}.
\end{proof}

\begin{lemma}
  \label{le:ga}
There exists a smooth function $\gamma: \R \to (0,\infty)$ such that 
\begin{align}\label{dbc}
\frac{|\gamma'(t)|}{\gamma(t)} &\leq C^\ast < \infty \qquad (t \in \R), \\ 
\label{bud}
0 < C_4 &\leq \frac{\|U(\cdot, t)\|_{L^\infty(\Om)}}{\gamma(t)} \leq C_5 < \infty \qquad (t \in \R) \,.
\end{align}
\end{lemma}

\begin{proof}
We follow \cite[Proof of Lemma 6.3]{p-H:rn}.
By \eqref{dno}, 
\begin{equation*}
|\log \|U(\cdot, k + 1)\|_{L^\infty(\Om)} - \log \|U(\cdot, k)\|_{L^\infty(\Om)}| \leq \frac{C_6}{2} \qquad (k \in \Z)\,.
\end{equation*}
It is therefore easy to find a smooth function $\eta : \R \to \R$ such that 
\begin{equation*}
\eta(k) := \log \|U(\cdot, k)\|_{L^\infty(\Om)} \quad (k \in \Z), \qquad |\eta'(t)| \leq C_6 \qquad (t \in \R)\,.
\end{equation*} 
Set $\gamma(t) := e^{\eta(t)}$. Then
\begin{equation*}
|\gamma'(t)| = \gamma(t) |\eta'(t)| \leq C_6 \gamma(t) \qquad (t \in \R)\,. 
\end{equation*}
Since $|\eta(t + \tau) - \eta(t)|
\leq C_6$ for each $\tau \in [0, 1]$,  we have
\begin{equation}\label{prf}
e^{-C_6} \leq e^{\eta(t) - \eta(t + \tau)} =  \frac{\gamma(t)}{\gamma(t + \tau)} \leq e^{C_6} \qquad (t \in \R, \tau \in[0, 1]) \,.  
\end{equation}
From \eqref{dno} and \eqref{prf} we next obtain 
\begin{equation*}
C_1 e^{-C_6} \leq \frac{\gamma(k)}{\gamma(t)} \frac{\|U(\cdot, t)\|_{L^\infty(\Om)}}{\|U(\cdot, k)\|_{L^\infty(\Om)}}  
\leq C_2e^{C_6} \qquad (t \in [k, k+1], k \in \Z) \,.
\end{equation*}
Since $\gamma(k) = \|U(\cdot, k)\|_{L^\infty(\Om)}$ for each $k \in
\Z$, 
\eqref{bud} follows.
\end{proof}

With $\gamma$ as in 
Lemma \ref{le:ga},
 set $Z(x, t) := \frac{U(x, t)}{\gamma(t)}$. Then 
\begin{equation*}
Z_t = \frac{U_t}{\gamma(t)} - \frac{\gamma'(t) U}{\gamma^2(t)} 
= \Delta Z + \frac{1}{\gamma(t)}f(x, \gamma(t)Z, \gamma(t)\grad Z) -
\frac{\gamma'(t)}{\gamma(t)} Z. 
\end{equation*}
Hence, $Z$ is a solution of the problem
  \begin{alignat}{2} \label{eq0Z}
Z_t &= \Delta Z + g(t, x, Z, \nabla Z),& \qquad &(x, t) \in \Om
\times \R\,, \\  
U &= 0, & \qquad &(x, t) \in \partial\Om \times (0, \infty)\,, \label{bc0Z}
\end{alignat}
where 
\begin{multline*}
g(t, x, u, p) := \frac{1}{\gamma(t)}f(x, \gamma(t)u, \gamma(t) p) - \frac{\gamma'(t)}{\gamma(t)} u \\
(x \in \Om, t \in \R, u \in [0,\infty), p \in \R^N) \,.
\end{multline*}
We verify that $g$ satisfies conditions (G1)-(G3) of Section \ref{linmove}.
Since $f$ is independent of $x_1$ and even in $p_1$, so is $g$ and
(G2) is satisfied.
Clearly, $g$ is continuous.  From (F1) and \eqref{dbc} we have 
\begin{multline*}
\sup_{x \in \bar{\Om},\, t \in \R}|g(x, t, u, p) - g(x, t, u', p')| \\
\leq  \sup_{x \in \bar{\Om},\, t \in \R} 
\frac{1}{\gamma(t)} |f(x, \gamma(t)u, \gamma(t)p) - f(x, \gamma(t)u', \gamma(t)p')| + \left| \frac{\gamma'(t)}{\gamma(t)}\right||u - u'| \\
\leq  \tilde{\beta}_0 (|u - u'| + |p - p'|)  \ \quad (u, u' \in \R_{+}, p, p' \in \R^N)\,,
\end{multline*}
for some $\tilde{\beta}_0 > 0$. Thus $g$ satisfies (G1).
Also, since $f(\cdot, 0, 0) \equiv 0$,
\begin{equation*}
g(t, x, 0, 0) = 0 \qquad ((x, t) \in \Om\times \R) \,, 
\end{equation*}
hence (G3) is satisfied as well.

Now, 
 $\|Z\|_{L^\infty(\Om \times\R)} \leq C_5$, by  \eqref{bud}. Hence
 $Z$ is a positive bounded  entire solution of \eqref{eq0Z},
 \eqref{bc0Z}, or, in the notation of Section \ref{linmove}, 
 $Z\in \cA^*$. 
Moreover, for each $x\in \Om$ relations \eqref{eq:harn} and 
\eqref{bud} give 
\begin{equation*}
 \inf_{t\in \R} Z(x,t)=\inf_{t\in \R} \frac{U(x,t)}{\|U(\cdot,
   t)\|_{L^\infty(\Om)}} 
\frac{\|U(\cdot, t)\|_{L^\infty(\Om)}}{\ga(t)} >0.
\end{equation*}
This implies that all functions in  $\al(Z)$ are 
(strictly) positive on $\Om$.  
Applying Theorem \ref{thmsurvey} to $Z$, 
we obtain $\La(Z(\cdot, t)) = 0$ for each $t\in\R$.
 Since, obviously, $\La(U(\cdot, t))
= \La(Z(\cdot, t))$, Lemma \ref{nul} is proved.

\section{Proof of Theorem \ref{thm1}}
\label{proof2}
Assume the hypotheses of Theorem \ref{thm1} to be satisfied. Recall
that $\om(u)$ is a compact subset of $C_0(\Om)$; we view it as
compact metric space with the induced norm (the supremum norm). 
Our first concern is to
show that there is a flow on $\om(u)$ defined by elements of
$\mathcal{A}$, that is, 
bounded entire solutions of problem \eqref{lim}. This is
not completely obvious, for under our assumptions 
one  cannot in general expect 
the initial-value problem for \eqref{eq0}, \eqref{bc0} 
to be well-posed in $C_0(\Om)$.

Fix any $z_0 \in \om(u)$.  Then there is 
$V \in \mathcal{A}$ such that $V(\cdot, 0) = z_0$
and $V(\cdot, t) \in \omega(u)$ for any $t \in \R$.
This can be proved by a straightforward
modification of the arguments given in the proof
of Lemma \ref{eqc}: in addition to
taking time intervals in $(0,\infty)$, rather than in 
$(-\infty,\infty)$, one needs to include the function 
$h$ in the linear nonhomogeneous
equation \eqref{eq:forlp}. 
Since $h$ is bounded, the regularity
estimates and the rest of the arguments go through
(note in particular that, thanks to  hypothesis (H), one obtains the
same limit autonomous equation as in \eqref{eq:nonhheat}).

Next we show that $V$ is
uniquely defined. Indeed, if  $V, \tilde V\in \cA$ satisfy  
$V(\cdot, 0) = z_0=\tilde V(\cdot, 0) $, then $w=V-\tilde V$
 is a solution of a
linear problem \eqref{prob-lin-e}, \eqref{prob-lin-bc}  
with $L \in \cE(  \beta_0, \Om \times \R)$ and $h \equiv 0$. Also $w(\codt,0)\equiv 0$. 
The maximum principle implies the uniqueness for the initial-boundary
value problem:  $w\equiv 0$ on $\Om\times [0,\infty)$. To prove
that $w\equiv 0$ on $\Om\times (-\infty,0]$ one uses 
 backward uniqueness for parabolic equations (see Remark \ref{rk:bu}).

In view of the uniqueness of $V$, setting 
$S_t z_0 := V(\cdot, t)$ for $t\in R$, 
we have defined a family $S$ of maps on
$\om(u)$.  We show that $S$ is a flow, that is,
\begin{itemize}
\item[(i)] $S_0$ is the identity on $\om(u)$,
\item[(ii)] $S_{t+ s} = S_t S_s$ \quad ($s, t\in \R)$, 
\item[(iii)] for each $t_0 \in \R$, the map $S_{t_0}$ is continuous.   
\end{itemize}
The fact that $S_0=I$ is obvious. The group property (ii) follows
from the uniqueness of $V$  and the time-translation invariance of
\eqref{lim}. To prove (iii), take  first $t_0 > 0$. 
For any $z_1, z_2 \in \omega (u)$, the function
 $w (\cdot, t) = S_t z_1 - S_t z_2$ is a solution of 
a linear problem \eqref{prob-lin-e}, \eqref{prob-lin-bc} on $\Om \times \R$
 with $L\in \mathcal{E}(\be_0, \Om \times \R)$ and $h \equiv 0$. By
Theorem \ref{max:principle},
\begin{equation*}
\| S_{t_0} z_1 - S_{t_0} z_2\|_{L^\infty(\Om)} \leq C(t_0) \| z_1 -  z_2\|_{L^\infty(\Om)} 
\end{equation*}
and the continuity of $S_{t_0}$ follows. Now let $t_0<0$. 
Properties (i) and (ii)  imply that  $S_{t_0}$ is the inverse to the
continuous map $S_{-t_0}$. Since $\om(u)$ is compact, the inverse is
continuous.

In the  next lemma, we show that $\omega(u)$ 
is \emph{chain transitive} under the flow $S$.
This means that for any $\phi, \psi \in \omega (u)$
and any $\varepsilon > 0$, $T > 0$ there exist
an integer $k \geq 1$, real numbers 
$t_1, \cdots, t_k \geq T$,
 and points
$\phi_0, \phi_1, \cdots, \phi_k \in \omega (u)$ with $\phi_0 = \phi$,
$\phi_k = \psi$, 
such that
\begin{equation}\label{jumpabst}
\|S_{t_{i + 1}} \phi_i - \phi_{i + 1}\|_{L^\infty(\Om)} < \varepsilon \qquad (0 \leq i < k) \,.
\end{equation}
This in particular means that $\omega(u)$ is 
\emph{chain recurrent}, that is, the above condition is satisfied with 
$\psi=\phi$, for any  $\phi\in \omega (u)$.

The following lemma is very similar to 
\cite[Lemma~7.5]{Chen-P:1}, \cite[Lemma~4.5]{p-Fo:conv-asympt} (see
also \cite{Mischaikow-S-T}); however, we cannot directly apply those
results here since the flow $S$ is not defined outside  $\omega(u)$.

\begin{lemma}\label{lem:ch:rec}
The set $\omega(u)$ is chain transitive under the flow $S$. 
\end{lemma}

\begin{proof}
Fix any $\varepsilon, T > 0$ and $\phi, \psi \in \omega(u)$. Denote $I= [T, 2T]$ and     
let $C_0 = C^\ast(N,  \be_0, T)$, where $C^\ast$ is as in Theorem 
\ref{max:principle}. 
By \eqref{oas} and (H) we can fix $T_1$ with
\begin{align}\label{co}
\dist_{C_0(\bar{\Om})}(u(\cdot, t), \omega (u)) &<  \frac{\varepsilon}{3 C_0} \quad (t \geq T_1) \,,
\\\label{co2}
\|h(\cdot, t)\|_{L^{\infty}(\Om)} &< \frac{\varepsilon}{3C_0}\quad (t \geq T_1) 
\,.
\end{align}
Since $\phi, \psi \in \omega (u)$, 
there are $s'_2 > s'_1 \geq T_1$ with $s'_2 - s'_1 > T$,
such that
$\|u(\cdot, {s'_1}) - \phi\|_{L^\infty(\Om)} < \frac{\varepsilon}{3}$ and 
$\|u(\cdot, {s'_2}) - \psi\|_{L^\infty(\Om)} < \frac{\varepsilon}{3}$.
Clearly, there exist
$k \in \N$ and an increasing finite
sequence $(s_i)_{i = 0}^k$ with $s_0 = s'_1$, $s_k = s'_2$, 
and $2T \geq s_{i + 1} - s_i \geq T$. 
As $s_i \geq s_1' \geq T_1$,  \eqref{co} implies
the existence of points  $\phi_i\in \omega (u)$, $i \in \{0,\dots,k \}$,
 with $\phi_0 = \phi$, 
$\phi_k = \psi$, and $\|\phi_i - u(\cdot, {s_i})\|_{L^\infty(\Om)} \leq \frac{\varepsilon}{3C_0}$.
We show that these points satisfy \eqref{jumpabst} with
 $t_i := s_{i} - s_{i - 1} \in [T, 2T]$. 
Indeed, 
\begin{equation*}
\|S_{t_{i + 1}} \phi_i - \phi_{i + 1}\|_{L^\infty(\Om)} 
\leq \|S_{t_{i + 1}} \phi_i - u(\cdot, s_{i + 1})\|_{L^\infty(\Om)} +
\|u(\cdot, s_{i + 1}) - \phi_{i + 1}\|_{L^\infty(\Om)} \,.
\end{equation*}
Now, the function $w_i(x, t) := S_{t}\phi_i(x) - u(x, s_i + t)$
satisfies 
\begin{alignat*}{2} 
(w_i)_t &= L_i(x, t) w_i + h(x, s_i + t),& \qquad &(x, t) \in \Om \times (0, \infty)\,, \\
w_i &= 0, & \qquad &(x, t) \in \partial\Om \times (0, \infty)\,, \\
w_i(\cdot, 0) &= u(\cdot, s_i) - \phi_i, & \qquad &x \in \Om\,,
\end{alignat*}
where  $L_i \in \mathcal{E}( \be_0, \Om \times (0, \infty))$. By Theorem \ref{max:principle}, \eqref{co}, and 
\eqref{co2}, 
\begin{align*} 
\|w_i(\cdot, t_{i + 1})\|_{L^\infty(\Om)} &\leq C_0 ( \|u(\cdot, s_i) - \phi_i\|_{L^\infty(\Om)} + 
\|h \|_{L^\infty(\Om \times (s_i, s_i + 2T))}) \\
&< C_0 \left( \frac{\varepsilon}{3C_0} + \frac{\varepsilon}{3C_0} \right)
 = \frac{2\varepsilon}{3}\,.
\end{align*}
By the definition of $\phi_{i}$ we obtain
\begin{equation*}
\|S_{t_{i + 1}} \phi_i - \phi_{i + 1}\|_{L^\infty(\Om)} 
\leq \|w_i(\cdot, t_{i + 1})\|_{L^\infty(\Om)} + \|u(\cdot, s_{i + 1}) - \phi_{i + 1}\|_{L^\infty(\Om)} 
< \varepsilon \,.
\end{equation*}
\end{proof}

We are now ready to  complete the proof of Theorem \ref{thm1}.

\begin{proof}[Proof of Theorem \ref{thm1}]
By Lemma \ref{efi}, the  set $E_+$ is finite.
Let $k$ be the number of elements of $E_+\cap \om(u)$.
We  write these elements  in the order of decreasing values of $\La$: 
\begin{equation*}
  E_+\cap \om(u)=\{z_1,\dots, z_k\}, \quad \La(z_1)\ge \dots\ge \La(z_k)
\end{equation*}
(we choose an arbitrary order among the elements with the same value
of $\La$). 

Now consider the following system of $k+1$ subsets of $\om(u)$: 
\begin{equation}
  \label{eq:morse}
  M_1:=\{z_1\}, \dots, M_k:=\{z_k\},\ M_{k+1}:=\{z\in \om(U):\La(z)=0\}. 
\end{equation}
By Theorem \ref{thm2}, for each $U\in \om(u) \setminus \bigcup_k M_{k} \subset \cA$ 
one has $\al(U)\subset
M_i$,  $\om(U)\subset
M_j$, for some $i>j$.   This means, in the terminology of 
\cite{Conley}, that the flow $S$ admits a
Morse decomposition with the Morse sets \eqref{eq:morse}.
By \cite[Theorem II.7.A]{Conley}, every chain recurrent set of $S$ is
a subset of the union the Morse set. Hence, by Lemma \ref{lem:ch:rec},
\begin{equation*}
  \om(u)\subset \bigcup_{j=1,\dots,k+1} M_j. 
\end{equation*}
Since  $\omega(u)$ is connected, it must be equal to one of the sets
$M_j$, which  gives the conclusion of Theorem \ref{thm1}.
\end{proof}

\section{Appendix A: Proof of Theorem \ref{thmsurvey}}
\label{appen}
Assume that $g$ is a function satisfying conditions (G1)-(G3) and 
$U$ is bounded (nonnegative) 
entire solution of \eqref{tdc}. Recall from Section \ref{linmove}
that the H\"older estimate  \eqref{eq:hol1U} holds and the trajectory
$\{U(\codt,t):t\in\R\}$ is relatively compact in $C_0(\Om)$.

Assume also that there is  $z_0\in \al(U)$ such that $z_0>0$ in
$\Om$. To prove Theorem \ref{thmsurvey}, we need to show that
 $\La(U(\cdot, t)) = 0$ for each $t \in \R$. 

Let
$\la_0 := \Lambda(z_0).$

We use the following lemma.
\begin{lemma}\label{cent}
  There is $z_1\in \al(U)$ 
such that $\La(z_1)\le \la_0$ and $V_\la z_1 > 0$ in 
$\Om_\la$ for each $\la\in [\la_0,\ell) \setminus \{0\}$. 
\end{lemma}

Suppose for a while that the statement in
 Lemma \ref{cent} is true. Let us show how it
implies the desired conclusion. 

By Lemma \ref{l:dec},  $\sigma :=\lim_{\tau\to-\infty}\La (U(\cdot,
\tau))$ satisfies $\La(U(\cdot,t))\le \sigma$ for each $t$ and 
\begin{equation} \label{eq:7a}
  \sigma\ge \La(z)\quad(z\in \al(U)).
\end{equation}
If $\la_0>0$,  then Lemma \ref{cent} in conjunction with Lemma
\ref{l:ref-2-moved}(ii) implies that $\sigma < \la_0$ in contradiction
to \eqref{eq:7a}. Thus $\la_0=0$. Now for each $\la>0$,
Lemmas \ref{cent} and
\ref{l:ref-2-moved}(ii) imply that $\sigma < \la$. Hence $\sigma=0$, and
consequently  $\La(U(\cdot, t)) = 0$ for each $t \in \R$.

It remains to prove the lemma.

\begin{proof}[Proof of Lemma \ref{cent}] 
There is a sequence  $t_n \to \infty$ such that
 $U(\cdot, -t_n) \to z_0$. Passing to a subsequence, we may also
 assume that $U(\cdot, -t_n+1)$ converges to some $z_1$ in $C_0(\Om)$
 (this follows by the compactness of the trajectory of $U$). Of
 course, $z_1\in \al(U)$. We show that  $z_1$ has the properties stated in
 the lemma. 

Pick any $\la \in [\la_0, \ell)\setminus \{0\}$.
 Since $z_0>0$ in
$\Om$,  we have $V_{\lambda} z_0 > 0$ on 
$\partial \Om_{\la} \cap \partial \Om$. 
Since $\Om$ is convex in $x_1$ (hypothesis (D1)), 
this clearly implies that
 $V_{\lambda} z_0 \not \equiv 0$ on any connected component of
$\Om_\la$. Also, $V_{\lambda} z_0 \ge  0$ on $\Om_\la$ as $\la\ge
\la_0=\La(z_0)$. 

Let now $G$ be any connected component of $\Om_\la$. The previous remarks
imply that there exist a 
ball $B_0 \subset G$ and $r_0 > 0$ such that 
$V_{\lambda} z_0 > 3r_0$ on $B_0$.  
Then $V_{\lambda}U(\cdot, -t_n) \geq 2r_0$ on $B_0$ 
for each sufficiently large $n$, and, by \eqref{eq:hol1U}, 
there exists $\vartheta \in  (0,1/4)$ 
independent of $n$ such that 
\begin{equation}\label{ball0}
V_{\lambda}U(\cdot, t) \geq r_0 \qquad ((x, t) \in B_0 \times [-t_n,-t_n+ 4\vartheta])\,.
\end{equation} 
We now show that  $V_{\la} z_1>0$ in $G$. It is sufficient to prove that 
$V_{\la} z_1 > 0$ in $D$ for any  any 
subdomain of $D\subset G$ such that 
$\bar D\subset G$ and $B_0\subset D$. Fix any such $D$.
We use  the following Harnack-type 
estimate on the function $V_\la U$ (recall that $w^\la=V_\la U$
is a solution of the linear problem  \eqref{ref}, \eqref{bcd}):
\begin{multline}
  \label{eq:har}
  V_\la U(x,-t_n+1)\ge  \sup_{D\times(-t_n+\vartheta,-t_n+2\vartheta)}
\kappa_1 \,(V_\la U)^+\\
- \sup_{\partial_P(G\times(-t_n,-t_n+1+\theta))} \kappa_2\,(V_\la U)^-
\qquad (x\in D).
\end{multline}
Here $\kappa_1$, $\kappa_2$ are positive constants independent of $n$
and $\partial_P$ stands for the parabolic boundary:
\begin{equation} \label{dpbd}
\partial_P(G\times(-t_n,-t_n+1+\theta)):=(\bar G\times
\{-t_n\})\cup (\partial G\times (-t_n,-t_n+1+\theta)).
\end{equation}
Estimate \eqref{eq:har} is a special case of an estimate given
in \cite[Lemma 3.5]{P:est}. 

Since $G$ is a connected component of $\Om_\la$, we 
have $\partial G\subset \partial\Om_\la$. Therefore, 
by \eqref{bcd}, $V_\la U\ge 0$ on  $\partial G\times \R$. Moreover,
since $V_\la U(\cdot,-t_n)\to V_\la z_0\ge 0$, uniformly  in
$\Om_\la$, the last term in \eqref{eq:har} approaches 0 as
$n\to\infty$. Using this and 
\eqref{ball0}, we obtain, upon passing to the limit in \eqref{eq:har},
that  $V_{\la} z_1\ge \kappa_1 r_1>0$ on $D$, as desired.

We have thus shown that $V_{\la} z_1>0$ in any connected component of
$\Om_\la$, hence $V_{\la} z_1>0$ in $\Om_\la$. Since 
$\la \in [\la_0, \ell)\setminus \{0\}$ was arbitrary, at the same time 
we have verified that $\La(z_1)\le \la_0$. 
The proof of the lemma is  complete.
\end{proof}

\section{Appendix B: Proof of (\ref{eq:regt})}
\label{appenB}
Assume that  (D1), (D2), (A), (F1), and (F2) hold and let  
$U$ be an arbitrary entire solution of   
\eqref{lim}. We verify that  the function
$\tilde U:t\mapsto U(\codt,t)$ satisfies
\begin{equation}
  \label{eq:regt1}
  \tilde U \in C(\R, H^1_0({\Om})) \cap L^2_{\textrm{loc}}(\R, D(\Delta))
\end{equation}
(see Remark \ref{rk:bu} for the meaning of  $D(\Delta)$).

First we rewrite \eqref{lim} as a linear nonhomogeneous problem
(cp. \eqref{eq:forlp}):  
\begin{alignat}{2}
    U_t &= L(x,t) U+f(x,0,0)\,, \qquad &&(x, t) \in
    \Om\times \R \,,\label{equc}\\ 
   U&=0 ,\qquad && (x, t) \in \partial\Om \times \R \,, \label{bcuc} 
\end{alignat}
where $L \in \cE(  \beta_0, \Om \times \R)$.  Note that 
since the principal part of $L$ is the Laplacian, \eqref{equc} can be
considered as an equation in the divergence form or nondivergence
form, as desired.  
We claim that $U$ is a weak solution of \eqref{equc}, \eqref{bcuc}. 
This is not completely obvious, even though 
$U$ is a classical solution, due to the lack of regularity of $\Om$. 
The nontrivial part of the claim is that 
$\tilde U\in L^2_{\textrm{loc}}(\R, H^1_0(\Om))$. 
We verify this by an approximation procedure.
Take a sequence of smooth domains $\Om_n\subset
\Om$ such that $\Om_n\subset \bar \Om_n\subset  \Om_{n+1}$ ($n=1,2,\dots$)
and  $\partial \Om_n$ approaches $\partial \Om$ in the
Hausdorff distance. Also, let $\eta_n:\R^N\to [0,1]$, $n=1,2,\dots,$  
be smooth functions  such that for  $n=2,3,\dots,$ one has 
$\eta_n\equiv 0$ on $\R^N\setminus \Om_n$ and 
$\eta_n\equiv 1$ on $\Om_{n-1}$. 

Fix any  $T\in (0,\infty)$.
On $\Om_n\times (-T,T)$, we solve the following 
initial-boundary value problem:
\begin{equation} \label{equcn}
 \begin{aligned}
    U^n_t &= L(x,t) U^n+f(x,0,0)\,,&& \qquad (x, t) \in
    \Om_n\times (-T,T) \,,\\ 
   U^n(x,t)&=0,&& \qquad (x, t) \in \partial\Om_n \times (-T,T)
   \,, \\
 U^n(x,-T)&=\eta_n(x)U(x,-T) ,&& \qquad x \in \Om_n 
   \,. 
\end{aligned} 
\end{equation}
There is a unique weak solution $U_n$ of 
\eqref{equcn} and, as $\Om_n$ is smooth, it coincides
with the unique strong solution (see 
\cite{Ladyzhenskaya-S-U, Lieberman:bk} for these concepts and
results). 
Now, $U^n-U$ is a strong solution of the linear equation
\begin{equation*}
    V_t = L(x,t) V\,, \quad (x, t) \in
    \Om_n\times (-T,T).
\end{equation*} 
By the maximum principle for strong solutions, 
there is a constant independent of $C$ such that 
\begin{equation}\label{maxpn}
  \|U-U^n\|_{L^\infty(\Om_n\times (-T,T))}\le C \|U^n - U\|_{L^\infty(\partial_P (\Om_n \times (-T, T))} \,,
\end{equation}
where $\partial_P$ stands for the parabolic boundary (cp. \eqref{dpbd}).
 Since $U \in C( \bar \Om\times [-T, T])$ and it satisfies the Dirichlet boundary condition,
the right hand side of \eqref{maxpn} converges to zero  as
$n\to\infty$. Thus, extending $U^n$ by zero outside $\Om_n\times
[-T,T)$, we have  $U-U^n\to 0$  in  $L^\infty(\Om \times (-T,T))$. At the
same time, since $0\le U^n\le U$ on the parabolic boundary of 
$\Om_n \times (-T,T)$, one can estimate the 
$L^2 ((-T,T), H^1_0({\Om_n}))$--norm of the function 
$\tilde U_n: t\mapsto U_n(\codt,t)$ 
by a constant independent of $n$ (see for example 
\cite[Theorem 6.1]{Lieberman:bk}).
Obviously, 
the $L^2 ((-T,T), H^1_0({\Om_n}))$--norm coincide with   
the $L^2 ((-T,T), H^1_0({\Om}))$--norm  for the extended function.  
Thus, passing to a subsequence if
necessary, we obtain 
that $\tilde U^n\to \tilde U$ 
weakly in $L^2 ((-T,T),
H^1_0({\Om}))$. Since $T$ was arbitrary, 
this gives us the desired conclusion that $\tilde U\in
L^2_{\textrm{loc}}(\R, H^1_0(\Om))$.
  
The Lipschitz continuity of $f(x,u,p)$ in $(u,p)$ (and the
continuity in $x$) now implies that the function 
$\phi(x, t)  := f(x, U(x, t), \nabla U(x, t))$ belongs to
$L^2(\Om\times (-T,T))$ for any $T>0$. Also, 
$\tilde U(\cdot,\tau)\in H^1_0(\Om)$ for almost all $\tau$. 
Pick any such $\tau\in (-\infty,0)$ and set $T:=-\tau$. 
The problem 
 \begin{alignat*}{2}
    V_t &= \De V+\phi(x,t)\,,&& \qquad (x, t)\in
    \Om\times (-T,T) \,,\\ 
   V(x,t)&=0,&& \qquad (x, t) \in \partial\Om \times (-T,T)
   \,,\\
 V(x,-T)&=U(x,-T) ,&& \qquad x \in \Om 
   \, 
\end{alignat*} 
has a unique weak solution $V$, and, as $U(\codt,-T)=U(\codt,\tau)\in
H^1_0(\Om)$ and $\phi\in L^2(\Om\times (-T,T))$, the function $\tilde
V:t\mapsto V(\codt,t)$ satisfies
\begin{equation}
  \label{eq:Vreg}
  \tilde V \in C([-T,T), H^1_0({\Om})) \cap L^2((-T,T), D(\Delta))
\end{equation}
(see for example \cite[Section II.3]{Temam:bk}). 
Using the fact that $\tilde U\in L^2((-T,T), H^1_0(\Om))$ one shows
easily that $U$ has to coincide with the weak solution $V$. Since
$T=-\tau$ can be taken arbitrarily large, we have verified that
\eqref{eq:regt1} holds. 

\bibliographystyle{amsplain}

\providecommand{\bysame}{\leavevmode\hbox to3em{\hrulefill}\thinspace}
\providecommand{\MR}{\relax\ifhmode\unskip\space\fi MR }
\providecommand{\MRhref}[2]{
  \href{http://www.ams.org/mathscinet-getitem?mr=#1}{#2}
}
\providecommand{\href}[2]{#2}
\def\cprime{$'$} \def\cprime{$'$} \def\cprime{$'$} \def\cprime{$'$}
  \def\cprime{$'$} \def\cprime{$'$} \def\cprime{$'$}

\end{document}